\documentclass[a4paper,notitlepage,10pt]{article}%
\usepackage{amssymb}
\usepackage{graphicx}
\usepackage{amsmath}
\usepackage{amsthm}
\usepackage{amsfonts}%
\providecommand{\U}[1]{\protect\rule{.1in}{.1in}}
\theoremstyle{plain}
\newtheorem{theorem}{Theorem}[section]

\newtheorem{corollary}[theorem]{Corollary}

\newtheorem{lemma}[theorem]{Lemma}

\newtheorem{proposition}[theorem]{Proposition}

\theoremstyle{definition}

\newtheorem{remark}[theorem]{Remark}
\numberwithin{equation}{section}
\numberwithin{theorem}{section}
\ifx\pdfoutput\relax\let\pdfoutput=\undefined\fi
\newcount\msipdfoutput
\ifx\pdfoutput\undefined\else
\ifcase\pdfoutput\else
\ifx\paperwidth\undefined\else
\ifdim\paperheight=0pt\relax\else\pdfpageheight\paperheight\fi
\ifdim\paperwidth=0pt\relax\else\pdfpagewidth\paperwidth\fi
\fi\fi\fi
\begin{document}

\title{Uniqueness and positivity issues in a quasilinear indefinite problem
\thanks{2020 \textit{Mathematics Subject Classification}. 35J25, 35J62,
35J92.} \thanks{\textit{Key words and phrases}. quasilinear, indefinite,
sublinear, uniqueness.} }
\author{Uriel Kaufmann\thanks{FaMAF-CIEM (CONICET), Universidad Nacional de
C\'{o}rdoba, Medina Allende s/n, Ciudad Universitaria, 5000 C\'{o}rdoba,
Argentina. \textit{E-mail address: }kaufmann@mate.uncor.edu} , Humberto Ramos
Quoirin \thanks{CIEM-FaMAF, Universidad Nacional de C\'{o}rdoba, (5000)
C\'{o}rdoba, Argentina. \textit{E-mail address: }humbertorq@gmail.com} ,
Kenichiro Umezu\thanks{Department of Mathematics, Faculty of Education,
Ibaraki University, Mito 310-8512, Japan. \textit{E-mail address:
}kenichiro.umezu.math@vc.ibaraki.ac.jp}
\and \noindent}
\maketitle

\begin{abstract}
We consider the problem
$$
-\Delta_{p}u=\lambda u^{p-1}+a(x)u^{q-1},\quad u\geq0\quad\mbox{ in }\Omega
,\quad\leqno{(P_\lambda)}
$$
under Dirichlet or Neumann boundary conditions. Here $\Omega$ is a smooth
bounded domain of $\mathbb{R}^{N}$ ($N\geq1$), $\lambda\in\mathbb{R}$,
$1<q<p$, and $a\in C(\overline{\Omega})$ changes sign. These conditions enable
the existence of dead core solutions for this problem, which may
admit multiple nontrivial
solutions. We show that for $\lambda<0$ the functional
\[
I_{\lambda}(u):=\int_{\Omega}\left(  \frac{1}{p}|\nabla u|^{p}-\frac{\lambda
}{p}|u|^{p}-\frac{1}{q}a(x)|u|^{q}\right)  ,
\]
defined in $X=W_{0}^{1,p}(\Omega)$ or $X=W^{1,p}(\Omega)$, has
\textit{exactly} one nonnegative global minimizer, and this one is the
\textit{only} solution of $(P_{\lambda})$ being positive in $\Omega_{a}^{+}$
(the set where $a>0$). In particular, this problem has at most one positive
solution for $\lambda<0$. Under some condition on $a$, the above uniqueness
result fails for some values of $\lambda>0$ as we obtain, besides the ground
state solution, a \textit{second} solution positive in $\Omega_{a}^{+}$. We
also provide conditions on $\lambda$, $a$ and $q$ such that these solutions
become positive in $\Omega$, and analyze the formation of dead cores for a
generic solution.

\end{abstract}

\section{Introduction and main results}

Let $\Omega$ be a bounded and smooth domain of $\mathbb{R}^{N}$ with $N\geq1$.
We deal with the problem
\[
\left\{
\begin{array}
[c]{lll}%
-\Delta_{p}u=\lambda u^{p-1}+a(x)u^{q-1} & \mathrm{in} & \Omega,\\
u\geq0 & \mathrm{in} & \Omega,\\
\mathbf{B}u=0 & \mathrm{on} & \partial\Omega,
\end{array}
\right.  \leqno{(P_{\lambda})}
\]
where $\Delta_{p}$ is the $p$-Laplacian operator, $\lambda\in\mathbb{R}$,
$a\in C(\overline{\Omega})$ changes sign and $1<q<p$, which is the so-called
\textit{subhomogeneous} (or \textit{sublinear} if $p=2$) case.

We consider either Dirichlet ($\mathbf{B}u=u$) or Neumann ($\mathbf{B}%
u=\partial_{\nu}u$, where $\nu$ is the outward unit normal to $\partial\Omega
$) homogeneous boundary conditions. Solutions of $(P_{\lambda})$ are
understood in the weak sense, i.e. as nonnegative critical points of the
energy functional
\[
I_{\lambda}(u):=\int_{\Omega}\left(  \frac{1}{p}|\nabla u|^{p}-\frac{\lambda
}{p}|u|^{p}-\frac{1}{q}a(x)|u|^{q}\right)  ,\quad u\in X,
\]
where $X=W_{0}^{1,p}(\Omega)$ in the Dirichlet case, and $X=W^{1,p}(\Omega)$
in the Neumann case. Standard regularity results for quasilinear elliptic
equations \cite{db,L} show that such solutions are in $C^{1,\alpha}%
(\overline{\Omega})$ for some $\alpha\in(0,1)$. If, in addition, $u>0$ in
$\Omega$, then we call it a \textit{positive solution} of $(P_{\lambda})$. We
are particularly interested in \textit{ground state} solutions (or least
energy solutions) of $(P_{\lambda})$, i.e. solutions $u$ such that
$I_{\lambda}(u)\leq I_{\lambda}(v)$ for any solution $v$ of $(P_{\lambda})$.

The most striking feature of $(P_{\lambda})$ under the above conditions on $a$ and $q$ is the possible occurence of
\textit{dead cores} (or \textit{free boundaries}), i.e. regions where
solutions vanish (see \cite{D} for a survey on this subject). As a matter of
fact \cite[Proposition 1.11]{D} shows that every solution of $\left(
P_{0}\right)  $ has a dead core if $a$ is too negative in some part of
$\Omega$ (see also \cite{bandle, BPT,KRQU16, KRQU3} for some examples of dead
core solutions with $p=2$). Whenever it occurs, this phenomenon provides a
rich structure to $S_{\lambda}$, the set of nontrivial solutions of
$(P_{\lambda})$, as a dead core solution may generate multiple elements in
$S_{\lambda}$ (see Remarks \ref{r0}(3) and \ref{rdc} below). This scenario is
considerably different from the Dirichlet one for $\left(  P_{0}\right)  $ and
$a>0$ (the \textit{strictly definite} case, see \cite{BO,DS}) or $a\geq0$ (the
\textit{definite} case, see \cite{IO}). Indeed, in view of the strong maximum
principle and the Hopf Lemma \cite{Va}, which apply in this situation, we have
$S_{\lambda}\subset\mathcal{P}^{\circ}$, where
\[
\mathcal{P}^{\circ}:=\left\{
\begin{array}
[c]{ll}%
\left\{  u\in C_{0}^{1}(\overline{\Omega}):u>0\ \mbox{in $\Omega$},\ \partial
_{\nu}u<0\ \mbox{on $\partial \Omega$}\right\}   &
\mbox{if ${\bf B }u= u$},\medskip\\
\left\{  u\in C^{1}(\overline{\Omega}%
):u>0\ \mbox{on $\overline{\Omega}$}\right\}   &
\mbox{if ${\bf B }u=\partial_\nu u$}.
\end{array}
\right.
\]
Since in this case $(P_{0})$ has at most one solution in $\mathcal{P}^{\circ}%
$, one deduces that $S_{0}$ is a singleton.

Note that for $q\geq p$ (the \textit{superhomogenous} and \textit{homogeneous}
cases) we also have $S_{\lambda}\subset\mathcal{P}^{\circ}$, even if $a$
changes sign (the \textit{indefinite} case). However, in such situation 
$S_{\lambda}$ is not a
singleton in general. We refer to \cite{AT,ALG,BCN,BCN2,BD,I,Ou} for an
overview of $(P_{\lambda})$ in the superhomogeneous indefinite case.

We shall proceed with the investigation carried out in \cite{KRQUpp} for
$(P_{0})$. For an account on this problem when
$p=2$, we refer to the recent review paper \cite{rimut}. In the sequel we
state and discuss our main results.

\subsection{Uniqueness for $\lambda<0$}

In \cite{KRQUpp} we proved that $I_{0}$ has a unique nonnegative global
minimizer, and this one turns out to be the only solution of $(P_{0})$
positive in
\[
\Omega_{a}^{+}:=\{x\in\Omega:a(x)>0\}.
\]
In particular, uniqueness of positive solutions holds for $(P_{0})$. We extend
this result to any $\lambda<0$.

\begin{theorem}
\label{t1} Assume that $\lambda<0$. Then $(P_{\lambda})$ has exactly one
ground state solution $U_{+}=U_{+}(\lambda)$, which has the following properties:

\begin{enumerate}
\item $U_{+}$ is the unique nonnegative global minimizer of $I_{\lambda}$.

\item $U_{+}$ is the only solution of $(P_{\lambda})$ such that $U_{+} >0$ in
$\Omega_{a}^{+}$.

\item The map $\lambda\mapsto U_{+}(\lambda)$ is increasing and continuous from $(-\infty,0)$ to $X$,  and
$U_{+}(\lambda)\rightarrow0$ in $X$ as $\lambda\rightarrow-\infty$.
\end{enumerate}
\end{theorem}

From Theorem \ref{t1}, we infer:

\begin{corollary}
\label{c1} Assume that $\lambda<0$. Then:

\begin{enumerate}
\item $(P_{\lambda})$ has at most one positive solution (which is $U_{+}$,
whenever it exists).

\item If $\Omega_{a}^{+}$ is connected then $U_{+}$ is the unique nontrivial
solution of $(P_{\lambda})$.
\end{enumerate}
\end{corollary}

\begin{remark}
\label{r0} \strut

\begin{enumerate}
\item Theorem \ref{t1} and Corollary \ref{c1} are, to the best of our
knowledge, new for $p\neq2$, for both Dirichlet and Neumann problems.
Furthermore, several results therein are new even for the Dirichlet problem
with $p=2$. We refer the reader to \cite[Theorem 1.1]{alama} and \cite{B}, for
previous results with $p=2$, in the Neumann and Dirichlet cases, respectively.

\item A natural question arising in connection with Theorem \ref{t1}(2) is
whether $U_{+}$ is a positive solution. Keeping $\lambda\leq0$, $q$, and
$a^{+}$ (the positive part of $a$) fixed, Theorem \ref{dc} below implies that
$U_{+}$ has a dead core if, roughly speaking, $a^{-}$ (the negative part of
$a$) is large enough . On the other hand, $U_{+}$ becomes a positive solution
as $a^{-}$ gets close to $0$, see Proposition \ref{pa}. The positivity of
$U_{+}$ also holds for $q$ close (and smaller than) $p$, and for $\lambda$
close (and smaller than) $0$ in the Neumann case if $\int_{\Omega}a>0$, cf.
Theorems \ref{t4} and \ref{t3} below.

\item When $\Omega_{a}^{+}$ is disconnected Remark \ref{rdc} below shows that,
for any $\lambda\leq0$, $(P_{\lambda})$ may have multiple solutions satisfying
$u\not \equiv 0$ in $\Omega_{a}^{+}$, so Theorem \ref{t1}(2) does \textit{not}
hold with `$U_{+}>0$' replaced by `$U_{+}\not \equiv 0$ '. In particular,
suppose that $\Omega_{a}^{+}=\displaystyle\cup_{i\in I}\Omega_{i}$, where
$\Omega_{i}$ are connected and open, and $J\subsetneq I$. Then $(P_{\lambda})$
may have several solutions positive in $\cup_{i\in J}\Omega_{i}$, so Theorem
\ref{t1}(2) does \textit{not} hold with `$\Omega_{a}^{+}$' replaced by
`$\cup_{i\in J}\Omega_{i}$'. Some concrete examples where these multiplicity
results appear for $\left(  P_{0}\right)  $ and
$p=2$ can be found in the proofs of \cite[Theorem 1.4(ii)]{NoDEA} and
\cite[Proposition 5.1]{KRQU3}. Let us also mention that these features were
already observed in \cite{bandle, BPT}.

\item The assertion in Theorem \ref{t1}(3) holds also in $(-\infty,0]$ for the Dirichlet problem, and for the Neumann one if $\int_\Omega a<0$. If $\int_\Omega a>0$ then $U_+(\lambda)$ blows up as $\lambda \to 0^-$, see Proposition \ref{pin}, Remark \ref{rin}, and Section 4.
\end{enumerate}
\end{remark}

Theorem \ref{t1} and Corollary \ref{c1} follow a long history of uniqueness
results for subhomogeneous and sublinear type problems, dating back at least
to \cite{KC}. The existence and uniqueness of a nontrivial solution for the
Dirichlet problem with $\lambda\leq0$ and $a>0$ in $\Omega$ has been
established in \cite{BO} for $p=2$ and in \cite{DS} for $p>1$ (see also
\cite{BK}). The indefinite case proves to be more delicate to handle, as shown
in \cite[Theorem 1.1]{alama}, \cite[Theorem 2.3]{bandle}, \cite[Lemma 3.1 and
Theorem 3.1]{BPT}, \cite[Theorem 2.1]{DSu}, which hold for $p=2$ (see also
\cite[Theorem A]{KRQUpp} for a precise description of these results when
$\lambda=0$). The proofs of these uniqueness results, which hold also for
sublinear nonlinearities that are not powerlike, are rather direct and make
use of the strong maximum principle. Here we shall proceed as in \cite{KRQUpp}
and follow an undirect strategy, relying on the combination of the following
results for $\lambda\leq0$:

\begin{itemize}
\item ground state solutions minimize $I_{\lambda}$ over the open set
\begin{equation}
\label{a+}\mathcal{A}_{+}:=\left\{  u\in X:\int_{\Omega}a(x)|u|^{q}>0\right\}
\end{equation}
(see Lemma \ref{l1});

\item there is a one-to-one correspondence between minimizers of $I_{\lambda}$
over $\mathcal{A}_{+}$ and nonnegative minimizers for $m_{+}=m_{+}(\lambda):=
\displaystyle \inf_{\mathcal{S}_{+}} E_{\lambda}$, where
\begin{equation}
E_{\lambda}(u):=\int_{\Omega}\left(  |\nabla u|^{p}-\lambda|u|^{p}\right)  ,
\label{de}%
\end{equation}
and
\[
\mathcal{S}_{+}=\mathcal{S}_{+}(a,q):=\left\{  u\in X:\int_{\Omega}a( x)
|u|^{q}=1\right\}
\]
(see Lemma \ref{l});

\item $m_{+}$ is achieved by \textit{exactly} one nonnegative minimizer (see
Proposition \ref{um}).

\item ground state solutions are the \textit{only} ones being positive in
$\Omega_{a}^{+}$ (see Proposition \ref{us}). This result is proved via a
generalized Picone's inequality (see Lemma \ref{lp}).
\end{itemize}

\subsection{A nonuniqueness result for $\lambda>0$}

We consider now $(P_{\lambda})$ with $\lambda>0$. Here the values
\[
\lambda_{1}:=\inf\left\{  \int_{\Omega}|\nabla u|^{p}:u\in X,\int_{\Omega
}|u|^{p}=1\right\}
\]
and
\[
\lambda^{\ast}=\lambda^{\ast}(a,q):=\inf\left\{  \int_{\Omega}|\nabla
u|^{p}:u\in X,\int_{\Omega}|u|^{p}=1,\int_{\Omega}a(x)|u|^{q}\geq0\right\}
\]
play important roles. It is readily seen that $\lambda_{1}$ and $\lambda^{*}$
are achieved. Moreover:

\begin{enumerate}
\item $\lambda_{1}$ is the first nonnegative eigenvalue of the problem
\[
-\Delta_{p}u=\lambda|u|^{p-2}u,\quad u\in X,
\]
and is achieved by a unique $\phi_{1}\in\mathcal{P}^{\circ}$ such that
$\int_{\Omega}\phi_{1}^{p}=1$.

\item $\lambda^{*}\geq\lambda_{1}$ with strict inequality if, and only if,
$\int_{\Omega}a\phi_{1}^{q}<0$.

\item If $X=W_{0}^{1,p}(\Omega)$ then $\lambda_{1}>0$. If $X=W^{1,p}(\Omega)$
then $\lambda_{1}=0$ and $\phi_{1}\equiv
|\Omega|^{-\frac{1}{p}}$. In this case the condition $\int_{\Omega}a\phi
_{1}^{q}<0$ reads $\int_{\Omega}a<0$.
\end{enumerate}

Let us give some insight on this case by briefly describing the geometry of
$I_{\lambda}$. Note that
\[
I_{\lambda}(u)=\frac{1}{p}E_{\lambda}(u)-\frac{1}{q}\int_{\Omega}a(x)|u|^{q},
\]
where $E_{\lambda}$ is given by \eqref{de}. Since $q<p$ and $E_{\lambda}$ is
coercive for $\lambda<\lambda_{1}$, one readily sees that $I_{\lambda}$ is
coercive, and consequently has a global minimizer $U_{+}$ for $\lambda
<\lambda_{1}$. One may easily show that $U_{+} \in\mathcal{A}_{+}$, i.e.
\[
I_{\lambda}(U_{+})=\min I_{\lambda}=\min_{\mathcal{A}_{+}}I_{\lambda}%
\quad\mbox{for}\quad\lambda<\lambda_{1}.
\]
More generally, any nontrivial solution of $(P_{\lambda})$ belongs to
$\mathcal{A}_{+}$ if $\lambda\leq\lambda_{1}$ (see Proposition \ref{puu}).

When $\lambda>\lambda_{1}$ the functional $I_{\lambda}$ becomes unbounded from
below, since $E_{\lambda}(t\phi_{1})\rightarrow-\infty$ as $t\rightarrow
\infty$. However, $E_{\lambda}$ is coercive in $\mathcal{A}_{+}$ for
$\lambda<\lambda^{\ast}$, so that $\min_{\mathcal{A} _{+}}I_{\lambda}$ is
achieved by some nonnegative minimizer, still denoted by $U_{+}$, for
$\lambda<\lambda^{\ast}$ (see Lemma \ref{l1}). Since $\mathcal{A}_{+}$ is
open, $U_{+}$ is a local minimizer of $I_{\lambda}$, and therefore a solution
of $(P_{\lambda})$. Recall that if $\int_{\Omega}a\phi_{1}<0$ then
$\lambda^{\ast}>\lambda_{1}$, so that in this case $\min_{\mathcal{A}_{+}%
}I_{\lambda}$ persists for $\lambda_{1}<\lambda<\lambda^{\ast}$. As already
observed, one may also obtain $U_{+}$ by minimizing $E_{\lambda}$ over the
$C^{1}$ manifold $\mathcal{S}_{+}$, up to some rescaling. When $\int_{\Omega
}a\phi_{1}^{q}<0$ and $\lambda_{1}<\lambda<\lambda^{\ast}$, a similar
procedure shall provide us with a second solution $U_{-}$ , since in this
case, setting
\[
\mathcal{S}_{-}=\mathcal{S}_{-}(a,q):=\left\{  u\in X:\int_{\Omega}a(
x)|u|^{q}=-1\right\}  ,
\]
one can show that $m_{-}=m_{-}(\lambda):= \displaystyle \inf_{\mathcal{S}_{-}}
E_{\lambda} $ is achieved and negative. The solution $U_{-}$ clearly belongs
to the negative counterpart of $\mathcal{A}_{+}$, namely
\[
\mathcal{A}_{-}:=\left\{  u\in X:\int_{\Omega}a\left(  x\right)
|u|^{q}<0\right\}  .
\]
More precisely, it belongs to
\[
E_{\lambda}^{-}:=\left\{  u\in X:E_{\lambda}(u)<0\right\}  ,
\]
which is a subset of $\mathcal{A}_{-}$ for $\lambda_{1}<\lambda<\lambda^{\ast
}$.

We shall see that $U_{-}>0$ in $\Omega_{a}^{+}$, and therefore uniqueness of
solutions positive in $\Omega_{a}^{+}$ \textit{breaks down} in this case. Let
us set%
\[
\Omega_{a}^{-}:=\{x\in\Omega:a(x)<0\}.
\]
The above discussion leads to the following result:

\begin{theorem}
\label{t2} Assume that $0<\lambda<\lambda^{\ast}$. Then:

\begin{enumerate}
\item $(P_{\lambda})$ has a ground state solution $U_{+}=U_{+}(\lambda)$.
Moreover, $U_{+}>0$ in $\Omega_{a}^{+}$ and $I_{\lambda}(U_{+}%
)=\displaystyle \min_{\mathcal{A}_{+}}I_{\lambda}$. In particular, $U_{+}$ is
a local minimizer of $I_{\lambda}$.

\item If $\int_{\Omega}a\phi_{1}^{q}<0$ and $\lambda>\lambda_{1}$ then
$(P_{\lambda})$ has a second solution $U_{-}=U_{-}(\lambda)$ satisfying
$U_{-}>0$ in $\Omega_{a}^{+}$. Moreover $U_{-}\not \equiv 0$ in $\Omega
_{a}^{-}$, and
\begin{equation}
\label{mm}I_{\lambda}(U_{-})=\min_{u\in E_{\lambda}^{-}}\max_{t>0}I_{\lambda
}(tu)=\displaystyle\inf_{u\neq0}\sup_{t>0}I_{\lambda}(tu).
\end{equation}

\end{enumerate}
\end{theorem}

\begin{remark}
\strut

\begin{enumerate}
\item The first equality in \eqref{mm} shows that $U_{-}$ can be seen as a
ground state solution over $\mathcal{A}_{-}$, since it has minimal energy
among solutions in this set (see Lemma \ref{mp}).

\item Theorem \ref{t2} is consistent with \cite[Theorem 1.1]{alama} and
\cite{B}, which deal with the Neumann and Dirichlet cases of $(P_{\lambda})$
when $p=2$, respectively. Let us mention that a multiplicity phenomenon
in $(\lambda_{1},\lambda_{1}+\delta)$,
for some $\delta>0$, also
occurs in the superhomogeneous and subcritical case $p<q\leq p^{\ast}$  under the condition $\int_{\Omega}a\phi_{1}^{q}<0$, cf.
\cite{AT, BCN,Ou} for $p=2$, and \cite{BD,I} for $p>1$. Here $p^{\ast}$
denotes the Sobolev critical exponent.

\item One may ask how large is $\lambda^{\ast}$ to estimate the domain of
existence for $U_{+}$ and $U_{-}$. In this regard, it is not difficult to see
that $\lambda^{\ast}\rightarrow\infty$ as the support of $a^{-}$ increases
towards $\Omega$ (see Remark \ref{lin}).
\end{enumerate}
\end{remark}

\subsection{Positivity vs dead core formation}

Whenever they exist, $U_{\pm}$ eventually become positive as $q$ approaches
$p$ (keeping $a$ and $\lambda$ fixed). This result is proved via a continuity
property (already used in \cite{KRQUpp}) showing that, up to some
normalization, $U_{\pm}$ converge to a nonnegative solution of $(P_{\lambda})$
with $q=p$. For this purpose we show that $\lambda^{\ast}(q)$ converges to
$\lambda^{\ast}(p)$ as $q\rightarrow p^{-}$ (see Proposition \ref{le}).

\begin{theorem}
\label{t4} Let $a$ and $\lambda<\lambda^{\ast}(a,p)$ be fixed. There exists
$q_{0} =q_{0}(\lambda,a)\in\lbrack1,p)$ such that $U_{+}\in\mathcal{P}^{\circ
}$ for $q_{0} <q<p$. If $\int_{\Omega}a\phi_{1}^{p}<0$ and $\lambda
_{1}<\lambda<\lambda^{\ast}(p)$ then the same conclusion holds for $U_{-}$. In
particular, in this case $(P_{\lambda})$ has two solutions in $\mathcal{P}%
^{\circ}$ for $q_{0}<q<p$.
\end{theorem}

\begin{remark}
The positivity of $U_{+}$ (respect. $U_{-}$) can also be deduced for $\lambda$
close and smaller than $\lambda_{1}$ (respect. larger than $\lambda_{1}$) if
$\int_{\Omega}a\phi_{1}^{p}>0$ (respect. $<0$), see Theorem \ref{t3}, and also
if $a^{-}$ is small enough, cf. Proposition \ref{pa}. To the best of our
knowledge, the existence of positive solutions provided by these results is
new for $p\neq2$ and $\lambda\neq0$.
\end{remark}

On the other hand, solutions of $(P_{0})$ develop a dead core as $a^{-}$
increases, cf. \cite[Proposition 1.11]{D} (see also \cite[Theorem E]{KRQUpp}
and \cite[Theorem 3.2]{ans} for a simpler statement with $p=2$). We show that
the dead core formation also occurs for other values of $\lambda$ and that in
certain situations the dead core approximates to $\Omega_{a}^{-}$ when $a^{-}$
tends to $\infty$.

\begin{theorem}
\label{dc} Assume that $a\in C(\overline{\Omega})$ changes sign,
$\Omega^{\prime}\Subset\Omega_{a}^{-}$, and $a_{n}:=a^{+}-na^{-}$. If either
$\lambda\leq0$ or $p=2$ and $\lambda\leq\lambda_{1}$, then there exists
$n_{0}>0$ such that any solution of $\left(  P_{\lambda,a_{n}}\right)  $
vanishes in $\Omega^{\prime}$ for all $n\geq n_{0}$.
\end{theorem}

Theorem \ref{dc} is proved by comparing solutions of $(P_{\lambda})$ with a
local barrier function obtained by a perturbation of $|x-x_{0}|^{\frac
{2p}{p-q}}$ with $x_{0}\in\Omega^{\prime}$. Note that this power is not the
usual one (see e.g. \cite[Proposition 1.11]{D}), namely, $|x-x_{0}|^{\frac
{p}{p-q}}$. This comparison is done using the weak comparison principle for
the operator $u\mapsto-\Delta_{p}u-\lambda u^{p-1}$ (defined on $W^{1,p}%
(\Omega)$), which breaks down when $\lambda>0$ and $p\not =2$ (see
\cite[Corollary 8]{garcia}). It is an interesting open question whether this
result still holds for $p\not =2$ and $\lambda>0$.

\begin{remark}
\label{rdc} Let either $\lambda\leq0$ or $p=2$ and $\lambda\leq\lambda_{1}$.
Using Theorem \ref{dc} one can find examples of solutions of $(P_{\lambda})$
vanishing in a prescribed number of connected components of $\Omega_{a}^{+}$.
This situation can be readily checked in the one-dimensional case: let
$\Omega:=(b,c)$ and $a\in C(\overline{\Omega})$ with $\Omega_{a}%
^{+}=(b,b+\frac{\delta}{2})\cup(c-\frac{\delta}{2},c)$, for some $\delta>0$
small. If $a$ is sufficiently negative in $(b+\frac{\delta}{2},c-\frac{\delta
}{2})$ then any solution of $(P_{\lambda})$ vanishes in $[b+\delta,c-\delta]$.
It follows that $U_{+}$ is a two-bumps solution, i.e. $U_{+}=U_{1}+U_{2}$,
with $U_{1},U_{2}\geq0$ in $\Omega$, $\text{supp }U_{1}\subset\lbrack
b,b+\delta]$, and $\text{supp }U_{2}\subset\lbrack c-\delta,c]$. Since $N=1$
(and thus $(P_{\lambda})$ holds pointwisely), we can easily check that $U_{1}$
and $U_{2}$ solve $(P_{\lambda})$. This procedure also applies if $\Omega
_{a}^{+}$ has finitely many connected components, to obtain a solution
vanishing in a prescribed number of these components, and also when $N>1$
provided that $p=2$ and $\lambda\leq\lambda_{1}$.
\end{remark}

\begin{figure}[ptbh]
\begin{center}
\includegraphics[scale=0.24]{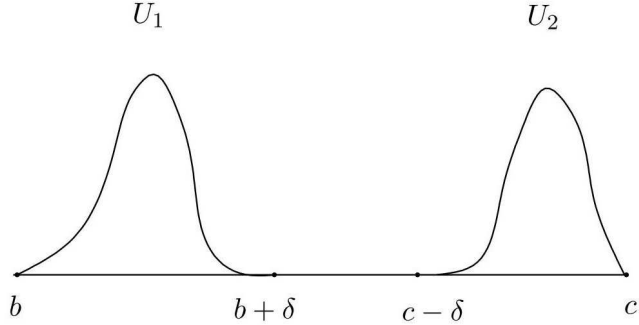}
\end{center}
\caption{A ground state solution having a dead core}%
\end{figure}

\subsection*{Final remarks}

One may check that besides Theorem \ref{t4}, the results above still hold if
$\Omega$ is a nonsmooth bounded domain. These results
also remain valid if $a\in L^{\infty}(\Omega)$,
if we set $\Omega_{a}^{\pm}$ as the largest open sets where
$a\gtrless0$ \textit{a.e.} (assuming in addition that $\mbox{essinf}_{\Omega
^{\prime}}a^{-}>0$ in Theorem \ref{dc}). Furthermore, the methods used in this
article allow us to derive similar results for the problem
\[
\left\{
\begin{array}
[c]{lll}%
-\Delta_{p}u=a(x)u^{q-1} & \mathrm{in} & \Omega,\\
u\geq0 & \mathrm{in} & \Omega,\\
|\nabla u|^{p-2}\partial_{\nu}u=\lambda u^{p-1} & \mathrm{on} & \partial
\Omega.
\end{array}
\right.
\]
We refer to \cite{KRQUR1,KRQUR2} for an investigation of this problem with
$p=2$.

\medskip

The outline of this article is the following: In Section 2 we show the
equivalence between ground state solutions and nonnegative minimizers of
$m_{+}$ and we prove Theorem \ref{t2}. Section 3 is devoted to the uniqueness
results in Theorem \ref{t1} and Corollary \ref{c1}. Section 4 is devoted to
positivity results, whereas in Section 5 we prove Theorem \ref{dc}.

\subsection*{Notation}

Throughout this paper, we use the following notation:

\begin{itemize}
\item $a^{\pm}:=\max(\pm a,0)$.

\item $\Omega_{f}^{+}:=\{x\in\Omega:f(x)>0\}$ and $\Omega_{f}^{-}:=\{x\in
\Omega:f(x)<0\}$, for $f\in C(\overline{\Omega})$.

\item Given $r>1$, we denote by $\Vert\cdot\Vert_{r}$ the usual norm in
$L^{r}(\Omega)$ and by $\Vert\cdot\Vert$ the usual norm in $X$, i.e. $\Vert
u\Vert=\Vert\nabla u\Vert_{p}$ if $X=W_{0}^{1,p}(\Omega)$ and $\left\Vert
u\right\Vert =\Vert\nabla u\Vert_{p}+\Vert u\Vert_{p}$ if $X=W^{1,p}(\Omega)$.

\item Strong and weak convergences are denoted by $\rightarrow$ and
$\rightharpoonup$, respectively.

\item Given $f\in L^{1}(\Omega)$, the integral $\int_{\Omega}f$ is considered
with respect to the Lebesgue measure. Equalities and inequalities involving
$f$ shall be understood holding \textit{a.e.}
\end{itemize}

\section{Ground states and constrained minimizers}

The next result will be used repeatedly throughout this paper:

\begin{lemma}
\label{la} \strut

\begin{enumerate}
\item The maps $u\mapsto\int_{\Omega}|u|^{p}$ and $u\mapsto\int_{\Omega
}a|u|^{q}$ are weakly continuous on $X$, i.e. if $u_{n}\rightharpoonup u$ in
$X$ then $\int_{\Omega}|u_{n}|^{p}\rightarrow\int_{\Omega}|u|^{p}$ and
$\int_{\Omega}a|u_{n}|^{q}\rightarrow\int_{\Omega}a|u|^{q}$.

\item Assume that $\lambda_{n}\rightarrow\lambda\leq\lambda^{\ast}$ with
$\int_{\Omega}a\phi_{1}^{q}>0$ if $\lambda=\lambda^{\ast}$. If $\{u_{n}%
\}\subset X$ is such that $\left\{  \int_{\Omega}a|u_{n}|^{q}\right\}  $ is
bounded and $\limsup E_{\lambda_{n}}(u_{n})<\infty$, then $\{u_{n}\}$ is
bounded in $X$.

\item For any $\lambda<\lambda^{\ast}$ there exists $C_{\lambda}=C_{\lambda
}(a)>0$ such that $E_{\lambda}(u)\geq C_{\lambda}\Vert u\Vert^{p}$ for every
$u\in\overline{\mathcal{A}_{+}}$ (the closure of $\mathcal{A}_{+}$ in $X$).

\item If $u$ and $v$ solve $(P_{\lambda})$ with $u\equiv tv$ for some $t>0$,
and $\int_{\Omega}au^{q}\not =0$, then $u\equiv v$.
\end{enumerate}
\end{lemma}

\begin{proof}
\strut
\begin{enumerate}
\item It follows from the compactness of the Sobolev embedding $X \subset
L^{t}(\Omega)$ for $t\in(1,p^{*})$.
\item If $\{u_{n}\}$ is unbounded in $X$ then we can assume that $\Vert
u_{n}\Vert\rightarrow\infty$, and $v_{n}:=\frac{u_{n}}{\Vert u_{n}\Vert
}\rightharpoonup v_{0}$ in $X$, for some $v_{0}$. It follows that
$\int_{\Omega}a|v_{0}|^{q}=\lim\int_{\Omega}a|v_{n}|^{q}=0$ and $E_{\lambda
}(v_{0})\leq\liminf E_{\lambda_{n}}(v_{n})\leq\limsup E_{\lambda_{n}}%
(v_{n})\leq0$. If $v_{0}\equiv0$ then we find that $E_{\lambda_{n}}%
(v_{n})\rightarrow0$, so that $v_{n}\rightarrow0$ in $X$, which is impossible.
Thus $v_{0}\not \equiv 0$. If $\lambda<\lambda^{\ast}$ then we find that
$E_{\lambda}(v_{0})>0$, a contradiction. Moreover, if $\int_{\Omega}a\phi
_{1}^{q}>0$ and $\lambda=\lambda^{\ast}=\lambda_{1}$, then $E_{\lambda_{1}%
}(v_{0})=0$, so that $v_{0}=c\phi_{1}$, for some $c\neq0$, which contradicts
$\int_{\Omega}a|v_{0}|^{q}=0$. Therefore $\{u_{n}\}$ is bounded in $X$.
\item First of all, note that $\overline{\mathcal{A}_{+}}=\{  u\in X:\int_{\Omega}a(x)|u|^{q} \geq 0\}$. Assume by contradiction that $\{u_{n}\}\subset \overline{\mathcal{A}_{+}}$ with
$E_{\lambda}(u_{n})\leq\frac{1}{n}\Vert u_{n}\Vert^{p}$. Then $v_{n}%
:=\frac{u_{n}}{\Vert u_{n}\Vert}$ satisfies $E_{\lambda}(v_{n})\leq\frac{1}%
{n}$ and $\int_{\Omega}a|v_{n}|^{q}\geq0$ for every $n$. We may assume, for some
$v_{0}\in X$, that $v_{n}\rightharpoonup v_{0}$ in $X$. Thus $E_{\lambda
}(v_{0})\leq0$ and $\int_{\Omega}a|v_{0}|^{q}\geq0$. Moreover, if $v_{0}%
\equiv0$ then $\limsup E_{\lambda}(v_{n})\leq E_{\lambda}(v_{0})$, so that
$v_{n}\rightarrow0$ in $X$, which is impossible. Hence $v_{0}\not \equiv 0$,
and from $E_{\lambda}(v_{0})\leq0$ we infer that $$\lambda\geq \frac{\int_\Omega |\nabla v_0|^p}{\int_\Omega v_0^p}\geq\lambda^{\ast},$$
a contradiction.
\item Since $u$ and $v$ solve $(P_\lambda)$, we have
$E_\lambda(u)=\int_\Omega au^q\neq 0$ and $E_\lambda(v)=\int_\Omega av^q$.
From $u \equiv tv$ we deduce, using the first equality, that
$t^{p-q}E_\lambda(v)=\int_\Omega av^q\neq 0$,
which clearly yields $t=1$.
\end{enumerate}
\end{proof}

\begin{remark}
Note that Lemma \ref{la}(3) says in particular that if $\lambda<\lambda^{\ast
}$ then $E_{\lambda}(u)>0$ for every $u\not \equiv 0$ such that $\int_{\Omega
}a|u|^{q}\geq0$. This result will be used repeatedly.
\end{remark}

Let us introduce
\[
M=M(\lambda):=\inf_{u\in\mathcal{A}_{+}}I_{\lambda}(u)
\]
and recall that%
\[
m_{\pm}=m_{\pm}(\lambda):=\inf_{v\in\mathcal{S}_{\pm}}E_{\lambda}(v).
\]

\begin{lemma}
\label{l1} Let $\lambda<\lambda^{*}$. Then:

\begin{enumerate}
\item There exists $U_{+}=U_{+}(\lambda)\in\mathcal{A}_{+}$ such that
$U_{+}\geq0$ and $\displaystyle I_{\lambda}(U_{+})=M<0$. Any such $U_{+}$ is a
ground state solution of $(P_{\lambda})$.

\item There exists $V_{\pm}=V_{\pm}(\lambda)\in\mathcal{S}_{\pm}$ such that
$V_{\pm}\geq0$ and $E_{\lambda}(V_{\pm})=m_{\pm}$. Moreover:

\begin{enumerate}
\item $m_{+}>0$ and $(m_{+})^{\frac{-1}{p-q}}V$ is a solution of $(P_{\lambda
})$ for any $V\geq0$ achieving $m_{+}$.

\item if $\int_{\Omega}a\phi_{1}^{q}<0$ and $\lambda>\lambda_{1}$ then
$m_{-}<0$ and $U_{-}:=(-m_{-})^{\frac{-1}{p-q}}V$ is a solution of
$(P_{\lambda})$ for any $V\geq0$ achieving $m_{-}$.


\end{enumerate}
\end{enumerate}
\end{lemma}

\begin{proof}
\strut
\begin{enumerate}
\item First of all, taking any $u\in\mathcal{A}_{+}$ we see that $I_{\lambda
}(tu)<0$ for $t>0$ small enough. Thus $M<0$. By Lemma \ref{la}(3), we know
that
\begin{equation}
I_{\lambda}(u)\geq C_{\lambda}\Vert u\Vert^{p}-C\Vert u\Vert^{q} \label{e1}%
\end{equation}
for some $C,C_{\lambda}>0$ and every $u\in\mathcal{A}_{+}$. Thus $I_{\lambda}$
is bounded from below in $\mathcal{A}_{+}$. Let us take $\{u_{n}%
\}\subset\mathcal{A}_{+}$ such that $I_{\lambda}(u_{n})\rightarrow M$. From
\eqref{e1} we see that $\{u_{n}\}$ is bounded, so we may assume that
$u_{n}\rightharpoonup u_{0}$ in $X$, so that $I_{\lambda}(u_{0})\leq M<0$,
since $I_{\lambda}$ is weakly lower semicontinuous. Moreover, $\int_{\Omega
}a|u_{0}|^{q}\geq0$ and $u_{0}\not \equiv 0$. Since $\lambda<\lambda^{\ast
}(q)$, we find that $E_{\lambda}(u_{0})\geq0$. If $\int_{\Omega}a|u_{0}%
|^{q}=0$ then we get $I_{\lambda}(u_{0})=E_{\lambda}(u_{0})\geq0$, a
contradiction. Therefore $u_{0}\in\mathcal{A}_{+}$ and $I_{\lambda}(u_{0})=M$.
Since $\mathcal{A}_{+}$ is open, we see that $u_{0}$ is a local minimizer, and
consequently a critical point of $I_{\lambda}$. As $I_{\lambda}(u)=I_{\lambda
}(|u|)$, we can chose $u_{0}$ nonnegative. We set $U_{+}:=u_{0}$. Finally, if
$u$ is a solution of $(P_{\lambda})$ and $u\not \in \mathcal{A}^{+}$ then
$I_{\lambda}(u)=\left(  \frac{1}{p}-\frac{1}{q}\right)  \int_{\Omega}%
au^{q}\geq0>I_{\lambda}(U_{+})$, i.e. $U_{+}$ is a ground state solution of
$(P_{\lambda})$.
\item First of all, since $\mathcal{S}_{+}\subset\mathcal{A}_{+}$, by Lemma
\ref{la}(3) we know that $m_{+}\geq0$. Moreover Lemma \ref{la}(2) (with
$\lambda_{n}=\lambda$ for all $n$) implies that $m_{-}\neq-\infty$ and any
minimizing sequence for $m_{\pm}$ is bounded in $X$. Since $E_{\lambda}$ is
weakly lower semi-continuous and $\mathcal{S}_{\pm}$ are weakly closed in $X$,
a standard compactness argument shows that there exist $V_{\pm}=V_{\pm
}(\lambda)\in\mathcal{S}_{\pm}$ such that $E_{\lambda}(V_{\pm})=m_{\pm}$.
Moreover since $E_{\lambda}(V_{\pm})=E_{\lambda}(|V_{\pm}|)$, we can take
$V_{\pm}\geq0$. Since $V_{+}\not \equiv 0$, we find by Lemma \ref{la}(3) that
$m_{+}>0$. If $\int_{\Omega}a\phi_{1}^{q}<0$ then $\left(  -\int_{\Omega}%
a\phi_{1}^{q}\right)  ^{-\frac{1}{q}}\phi_{1}\in\mathcal{S}_{-}$, and if
$\lambda>\lambda_{1}$ then $E_{\lambda}\left(  \phi_{1}\right)  <0$, so that
$m_{-}\leq\frac{E_{\lambda}(\phi_{1})}{\left(  -\int_{\Omega}a\phi_{1}%
^{q}\right)  ^{\frac{p}{q}}}<0$. Finally, by the Lagrange multipliers rule,
$V_{\pm}$ solve
\[
-\Delta_{p}V_{\pm}=\lambda V_{\pm}^{p-1}+\alpha_{\pm}aV_{\pm}^{q-1}%
\]
for some $\alpha_{\pm}\in\mathbb{R}$. It follows that $E_{\lambda}(V_{\pm
})=\pm\alpha_{\pm}$, i.e. $\alpha_{\pm}=\pm m_{\pm}$. Thanks to the
homogeneity of the equation above, we see that $(\pm m_{\pm})^{\frac{-1}{p-q}%
}V_{\pm}$ solve $(P_{\lambda})$. Since the above argument applies to any
$V\geq0$ achieving $m_{+}$ or $m_{-}$, this concludes the proof.
\end{enumerate}
\end{proof}

\begin{remark}
\label{nem} The condition $\lambda<\lambda^{*}$ is nearly optimal for $M$ to
be achieved and $m_{+}$ to be positive, as $M=-\infty$ and $m_{+}<0$ for
$\lambda>\lambda^{*}$. Indeed, let $u_{0}\geq0$ achieve $\lambda^{*}$, so that
$E_{\lambda}(u_{0})<0$ for $\lambda>\lambda^{*}$. We choose some nontrivial
$\psi\geq0$ supported in $\Omega_{a}^{+}$, and set $u_{s}=u_{0}+s\psi$. Then
$\int_{\Omega}au_{s}^{q}>0$ and $E_{\lambda}(u_{s})<0$ for $s>0$ small enough.
It follows that $u_{s} \in\mathcal{A}_{+}$, so that $tu_{s} \in\mathcal{A}%
_{+}$, and $I_{\lambda}(tu_{s})=\frac{t^{p}}{p}E_{\lambda}(u_{s})-\frac{t^{q}%
}{q}\int_{\Omega}au_{s}^{q} \to-\infty$ as $t \to\infty$, which yields the
first assertion. The second one follows from $m_{+} \leq\frac{E_{\lambda
}(u_{s})}{\left(  \int_{\Omega}au_{s}^{q}\right)  ^{\frac{p}{q}}}<0$.
\end{remark}

\begin{lemma}
\label{l} Let $\lambda<\lambda^{*}$. Then:

\begin{enumerate}
\item If $M$ is achieved by $U$ then $m_{+}$ is achieved by $\left(
\int_{\Omega}aU^{q}\right)  ^{-\frac{1}{q}}U$.

\item If $m_{+}$ is achieved by $V$ then $M$ is achieved by $tV$, for some
$t>0$.

\item If $m_{+}$ is achieved by $V_{+} \geq0$ then $V_{+}>0$ in $\Omega
_{a}^{+}$. In a similar way, if $m_{-}< 0$ then $V_{-}>0$ in $\Omega_{a}^{+}$,
and moreover $V_{-} \not \equiv 0$ in $\Omega_{a}^{-}$.

\item If $M$ is achieved by $U\geq0$ then $U>0$ in $\Omega_{a}^{+}$.
\end{enumerate}
\end{lemma}

\begin{proof}
\strut
\begin{enumerate}
\item Let $U\in\mathcal{A}_{+}$ be such that $I_{\lambda}(U)=M$ and
$V\in\mathcal{S}_{+}$ be such that $E_{\lambda}(V)=m_{+}$. Then, for any
$t\in\mathbb{R}$,
\begin{align*}
\frac{1}{p}E_{\lambda}(U)-\frac{1}{q}\int_{\Omega}a|U|^{q} &  =I_{\lambda
}(U)\leq I_{\lambda}(tV)\\
&  =\frac{t^{p}}{p}E_{\lambda}(V)-\frac{t^{q}}{q}\int_{\Omega}a|V|^{q}%
=\frac{t^{p}}{p}m_{+}-\frac{t^{q}}{q}.
\end{align*}
We choose $t=t_{0}:=\left(  \int_{\Omega}a|U|^{q}\right)  ^{\frac{1}{q}}$, so
that
\[
\frac{1}{p}E_{\lambda}(U)-\frac{t_{0}^{q}}{q}\leq\frac{t_{0}^{p}}{p}%
m_{+}-\frac{t_{0}^{q}}{q}.
\]
Thus $E_{\lambda}(t_{0}^{-1}U)\leq m_{+}$, which yields the desired conclusion.
\item We use a similar trick. Let $U\in\mathcal{A}_{+}$ be such that
$I_{\lambda}(U)=M$, and set $t_{0}$ as in (1). Note that
\[
I_{\lambda}(t_{0}V)=\frac{t_{0}^{p}}{p}E_{\lambda}\left(  V\right)
-\frac{t_{0}^{q}}{q}\leq\frac{t_{0}^{p}}{p}E_{\lambda}\left(  t_{0}%
^{-1}U\right)  -\frac{t_{0}^{q}}{q}=\frac{E_{\lambda}(U)}{p}-\frac{t_{0}^{q}%
}{q}=M,
\]
and since $t_{0}V\in\mathcal{A}_{+}$, we have $I_{\lambda}(tV)=M$.
\item Assume by contradiction that $V_{\pm}(x_{0})=0$ for some $x_{0}\in
\Omega_{a}^{+}$. Then, by the strong maximum principle \cite{Va}, $V_{\pm}\equiv0$ in a
ball $B\subset\Omega_{a}^{+}$. We choose then $\phi\in C_{0}^{\infty}(B)$ with
$\phi\geq0,\not \equiv 0$, and extend it by zero to $\Omega$. Note that
$V_{\pm}+t\phi\in\mathcal{A}_{\pm}$ for $t>0$ small enough. Moreover,
\[
\frac{E_{\lambda}(V_{\pm}+t\phi)}{\left(  \pm\int_{\Omega}a|V_{\pm}+t\phi
|^{q}\right)  ^{\frac{p}{q}}}=\frac{m_{\pm}+t^{p}E_{\lambda}(\phi)}{\left(
1\pm t^{q}\int_{\Omega}a\phi^{q}\right)  ^{\frac{p}{q}}}<m_{\pm}%
\]
for $t>0$ small enough, providing us with a contradiction. Indeed, note that
the above inequality holds if, and only if,
\[
E_{\lambda}(\phi)<m_{\pm}\frac{\left(  1\pm t^{q}\int_{\Omega}a\phi
^{q}\right)  ^{\frac{p}{q}}-1}{t^{p}},
\]
which holds for $t>0$ small enough, since $m_{\pm}\gtrless0$ and
\[
\lim_{t\rightarrow0^{+}}\frac{\left(  1\pm t^{q}\int_{\Omega}a\phi^{q}\right)
^{\frac{p}{q}}-1}{t^{p}}=\lim_{t\rightarrow0^{+}}\frac{\left(  1\pm t^{q}%
\int_{\Omega}a\phi^{q}\right)  ^{\frac{p}{q}-1}\left(  \pm\int_{\Omega}%
a\phi^{q}\right)  }{t^{p-q}}=\pm\infty.
\]
Finally, it is clear that $V_{-}\not \equiv 0$ in $\Omega_{a}^{-}$ since
$V_{-}\in\mathcal{S}_{-}$.
\item It follows directly from items (1) and (3).
\end{enumerate}
\end{proof}

\begin{lemma}
\label{mp} Assume that $\int_{\Omega}a\phi_{1}^{q}<0$ and $\lambda_{1}%
<\lambda<\lambda^{\ast}$. If $m_{-}$ is achieved by $V$ then $I_{\lambda
}((-m_{-})^{\frac{-1}{p-q}}V)=\displaystyle\inf_{u\not \equiv 0}\sup
_{t>0}I_{\lambda}(tu)=\min_{u\in E_{\lambda}^{-}}\max_{t>0}I_{\lambda}(tu)>0$.
In particular, $(-m_{-})^{\frac{-1}{p-q}}V$ has the least energy among
solutions in $\mathcal{A}_{-}$.
\end{lemma}

\begin{proof}
First of all, note that $\sup_{t>0}I_{\lambda}(tu)=\infty$ if $u\not \equiv 0$
and $E_{\lambda}(u)\geq0$. Thus we shall consider this supremum for $u\in
E_{\lambda}^{-}$, in which case $u\in\mathcal{A}_{-}$ (since $\lambda
<\lambda^{\ast}(q)$) and $\sup_{t>0}I_{\lambda}(tu)=\max_{t>0}I_{\lambda}%
(tu)>0$. Set $i_{u}(t):=I_{\lambda}(tu)$. Then $i_{u}$ has a unique global
maximum point $t_{0}(u):=\left(  \frac{\int_{\Omega}a|u|^{q}}{E_{\lambda}%
(u)}\right)  ^{\frac{1}{p-q}}$ and
\begin{equation}
\max_{t>0}I_{\lambda}(tu)=i_{u}(t_{0}(u))=\left(  \frac{1}{q}-\frac{1}%
{p}\right)  \frac{(-\int_{\Omega}a|u|^{q})^{\frac{p}{p-q}}}{(-E_{\lambda
}(u))^{\frac{q}{p-q}}}.\label{max}%
\end{equation}
On the other hand, set $U:=t_{0}\left(  V\right)  V$. Note that $t_{0}%
(V)=\left(  -m_{-}\right)  ^{\frac{-1}{p-q}}$. From (\ref{max}) we see that%
\[
I_{\lambda}(U)=I_{\lambda}(t_{0}\left(  V\right)  V)=\max_{t>0}I_{\lambda
}(tV)=\left(  \frac{1}{q}-\frac{1}{p}\right)  (-m_{-})^{\frac{-q}{p-q}}.
\]
Now, if $u\in\mathcal{A}_{-}$ then $m_{-}\leq\frac{E_{\lambda}(u)}%
{(-\int_{\Omega}a|u|^{q})^{\frac{p}{q}}}$, so that
\[
(-m_{-})^{\frac{-q}{p-q}}\leq\frac{(-\int_{\Omega}a|u|^{q})^{\frac{p}{p-q}}%
}{(-E_{\lambda}(u))^{\frac{q}{p-q}}}.
\]
Thus $I_{\lambda}(U)\leq\max_{t>0}I_{\lambda}(tu)$ for every $u\in E_{\lambda
}^{-}$. In addition, equality holds when $u=U$.
Finally,  if $u \in \mathcal{A}_-$ solves $(P_\lambda)$ then $u \in E_\lambda^-$ and $I_\lambda(u)=\max_{t>0}I_{\lambda}(tu)$, so
$I_\lambda(U)\leq I_\lambda(u)$, as claimed.
\end{proof}

\begin{remark}
\label{lin} Let $\{a_{n}\}$ be a sequence such that $a_{n} \to a$ in
$C(\overline{\Omega})$ with $a<0$ in $\Omega$. Then $\lambda^{*}(a_{n})
\to\infty$. Indeed, otherwise we can assume that $\lambda^{*}(a_{n})$ remains
bounded. It follows that if $u_{n}$ achieve $\lambda^{*}(a_{n})$ then
$\{u_{n}\}$ is bounded in $X$ and we can assume that $u_{n} \rightharpoonup u$
in $X$. Thus $\int_{\Omega}|u|^{p}=1$ and $\int_{\Omega}a|u|^{q} =\lim
\int_{\Omega}a_{n}|u_{n}|^{q} \geq0$. But since $a<0$ in $\Omega$ we must have
$u \equiv0$, contradicting $\int_{\Omega}|u|^{p}=1$.
\end{remark}

\section{Uniqueness results}

The following result extends \cite[Proposition 2.3]{KRQUpp}, which applies to
the case $\lambda=0$ and is just a reformulation of \cite[Theorem 1.1]{KLP}.
We rely here on a \textit{hidden convexity} property, namely, the fact that
$|\nabla u|^{p}$ is convex along the path $\left(  tu^{q}+(1-t)v^{q}\right)
^{\frac{1}{q}}$ with $t \in[0,1]$ and $0\leq u,v \in X$, cf. \cite[Proposition
2.6]{BF}. \newline

\begin{proposition}
\label{um} If $\lambda<0$ then $m_{+}$ is achieved by exactly one $V_{+}
\geq0$.
\end{proposition}

\begin{proof}
We proceed as in \cite[Proposition 2.3]{KRQUpp}. Assume that
$V_{1},V_{2}\geq0$, $V_{1},V_{2}\in\mathcal{S}_{+}$ and $E_{\lambda}%
(V_{1})=E_{\lambda}(V_{2})=m_{+}$. Then $W:=\left(  \frac{V_{1}^{q}+V_{2}^{q}%
}{2}\right)  ^{\frac{1}{q}}$ satisfies $W\in\mathcal{S}_{+}$, and (as shown in
\cite[Proposition 2.3]{KRQUpp})
\[
|\nabla W|^{p}\leq\frac{1}{2}\left(  |\nabla V_{1}|^{p}+|\nabla V_{2}%
|^{p}\right)  ,
\]
with strict inequality in the set
\[
F:=\{x\in\Omega_{V_{1}}^{+}\cup\Omega_{V_{2}}^{+}:V_{1}(x)\neq V_{2}%
(x),|\nabla V_{1}(x)|+|\nabla V_{2}(x)|\neq0\}.
\]
Moreover, by convexity we have
\[
W^{p}=\left(  \frac{V_{1}^{q}+V_{2}^{q}}{2}\right)  ^{\frac{p}{q}}\leq
\frac{V_{1}^{p}+V_{2}^{p}}{2}.
\]
Hence, since $\lambda\leq0$, we find that
\[
E_{\lambda}(W)\leq\frac{1}{2}\left(  E_{\lambda}(V_{1})+E_{\lambda}%
(V_{2})\right)  =m_{+},
\]
with strict inequality if $|F|>0$. Thus $|F|=0$, so that for almost every
$x\in\Omega$ we have
\[
V_{1}(x)=V_{2}(x)\quad\text{or}\quad\nabla V_{1}(x)=\nabla V_{2}(x)=0.
\]
In particular, $\nabla V_{1}=\nabla V_{2}$ a.e. in $\Omega$, so $V_{1}\equiv
V_{2}+C$, for some constant $C$. If $C\neq0$ then, from the alternative above,
we have $\nabla V_{1}=\nabla V_{2}=0$ a.e. in $\Omega$, which is impossible.
Therefore $V_{1}\equiv V_{2}$, and the proof is complete.
\end{proof}

\begin{corollary}
\label{un} If $\lambda<0$ then $M$ is achieved by exactly one $U_{+}%
=U_{+}(\lambda)\geq0$, namely, $U_{+} \equiv m_{+}^{-\frac{1}{p-q}}V_{+}$,
where $V_{+}$ is the unique nonnegative minimizer associated to $m_{+}$.
\end{corollary}

\begin{proof}
If $U\geq 0$ achieves $M$, then $U$ solves $(P_\lambda)$. Moreover, by Lemma \ref{l}(1) we know that $\left(  \int_{\Omega
}aU^{q}\right)  ^{-\frac{1}{q}}U$ achieves $m_{+}$ and so, from Proposition \ref{um}, we have $U\equiv tV_+$, with $t=\left(  \int_{\Omega
}aU^{q}\right)  ^{\frac{1}{q}}$. On the other hand, by Lemma \ref{l1}(2-a) we know that $m_{+}^{-\frac{1}{p-q}}V_{+}$ also solves $(P_\lambda)$. Since $\int_\Omega aU^q>0$, Lemma \ref{la}(4) yields that $U \equiv m_{+}^{-\frac{1}{p-q}}V_{+}$.
\end{proof}

\begin{remark}
For $0<\lambda<\lambda^{\ast}$ we still have a nonnegative minimizer $V_{+}$
for $m_{+}$, but we don't know if this one is unique. In any case $U_{+}%
=m_{+}^{\frac{-1}{p-q}}V_{+}$ still provides a minimizer for $M$.
\end{remark}

Let us show that $U_{+}$ is the only solution of $(P_{\lambda})$ satisfying
$U_{+}>0$ in $\Omega_{a}^{+}$ when $\lambda<0$. We follow the argument of
\cite[Theorem 5.1]{BF}, which deals with the case $a \equiv1$ and $\lambda=0$.
The following two inequalities shall be used:

\begin{lemma}
[Generalized Picone's identity]\label{lp}If $u,v\in W^{1,p}(\Omega)$ with
$u>0$ and $v\geq0$ then
\[
|\nabla u|^{p-2}\nabla u\nabla\left(  \frac{v^{q}}{u^{q-1}}\right)
\leq|\nabla u|^{p-q}|\nabla v|^{q}.
\]

\end{lemma}

The previous lemma is a particular case of \cite[Proposition 2.9]{BF}.

\begin{lemma}
\label{in} We have $b^{1-t}c^{t}+d^{1-t}e^{t}\leq(b+d)^{1-t}(c+e)^{t}$ for any
$b,c,d,e>0$ and $t \in[0,1]$.
\end{lemma}

\begin{proof}
Note that setting $x:=\frac{b}{d}$ and $y:=\frac{c}{e}$ the above inequality
reads $x^{1-t}y^{t}+1\leq(x+1)^{1-t}(y+1)^{t}$. We fix $y$ and $t$, and set
$f(x):=(x+1)^{1-t}(y+1)^{t}-x^{1-t}y^{t}-1$. One may easily check that $f$ has
a global minimum at $x=y$, and $f(y)=0$, which yields the desired inequality.
\end{proof}

\begin{proposition}
\label{us} If $\lambda<0$ then $U_{+}$ is the only solution of $(P_{\lambda})$
such that $U_{+}>0$ in $\Omega_{a}^{+}$.
\end{proposition}

\begin{proof}
Let $u$ be a solution of $(P_{\lambda})$ such that $u>0$ in $\Omega_{a}^{+}$,
and $\epsilon>0$. We take $\frac{V_{+}^{q}}{(u+\epsilon)^{q-1}}$ as test
function in $(P_{\lambda})$, so that
\[
\int_{\Omega}a\left(  \frac{u}{u+\epsilon}\right)  ^{q-1}V_{+}^{q}%
\ +\ \lambda\int_{\Omega}u^{p-1}\frac{V_{+}^{q}}{(u+\epsilon)^{q-1}}%
=\int_{\Omega}|\nabla u|^{p-2}\nabla u\nabla\left(  \frac{V_{+}^{q}%
}{(u+\epsilon)^{q-1}}\right)  .
\]
By Lemma \ref{lp} (with $\sigma=q$ and $u+\epsilon$ instead of $u$) we have
\[
\int_{\Omega}a\left(  \frac{u}{u+\epsilon}\right)  ^{q-1}V_{+}^{q}+\lambda
\int_{\Omega}u^{p-1}\frac{V_{+}^{q}}{(u+\epsilon)^{q-1}}\leq\int_{\Omega
}|\nabla u|^{p-q}|\nabla V_{+}|^{q}.
\]
Note that $\frac{u}{u+\epsilon}\rightarrow\chi_{\Omega_{u}^{+}}$ as
$\epsilon\rightarrow0$, where $\chi$ denotes the
characteristic function. Thus, by
Lebesgue's dominated convergence theorem, we find that
\[
\int_{\Omega_{u}^{+}}aV_{+}^{q}+\lambda\int_{\Omega}u^{p-q}V_{+}^{q}\leq
\int_{\Omega}|\nabla u|^{p-q}|\nabla V_{+}|^{q}.
\]
Now, by Holder's inequality we find that
\[
\int_{\Omega}|\nabla u|^{p-q}|\nabla V_{+}|^{q}\leq\left(  \int_{\Omega
}|\nabla u|^{p}\right)  ^{\frac{p-q}{p}}\left(  \int_{\Omega}|\nabla
V_{+}|^{p}\right)  ^{\frac{q}{p}}%
\]
and
\[
\int_{\Omega}u^{p-q}V_{+}^{q}\leq\left(  \int_{\Omega}u^{p}\right)
^{\frac{p-q}{p}}\left(  \int_{\Omega}V_{+}^{p}\right)  ^{\frac{q}{p}},
\]
so that
\[
\int_{\Omega_{u}^{+}}aV_{+}^{q}\leq\left(  \int_{\Omega}|\nabla u|^{p}\right)
^{\frac{p-q}{p}}\left(  \int_{\Omega}|\nabla V_{+}|^{p}\right)  ^{\frac{q}{p}%
}-\lambda\left(  \int_{\Omega}u^{p}\right)  ^{\frac{p-q}{p}}\left(
\int_{\Omega}V_{+}^{p}\right)  ^{\frac{q}{p}}.
\]
In addition, since $u>0$ in $\Omega_{a}^{+}$, we have $a\leq0$ in
$\Omega\setminus\Omega_{u}^{+}$, which implies that
\[
\int_{\Omega_{u}^{+}}aV_{+}^{q}=1-\int_{\Omega\setminus\Omega_{u}^{+}}%
aV_{+}^{q}\geq1,
\]
and therefore
\[
1\leq\left(  \int_{\Omega}|\nabla u|^{p}\right)  ^{\frac{p-q}{p}}\left(
\int_{\Omega}|\nabla V_{+}|^{p}\right)  ^{\frac{q}{p}}+\left(  -\lambda
\int_{\Omega}u^{p}\right)  ^{\frac{p-q}{p}}\left(  -\lambda\int_{\Omega}%
V_{+}^{p}\right)  ^{\frac{q}{p}}.
\]
From Lemma \ref{in}, it follows that
\[
1\leq E_{\lambda}(u)^{\frac{p-q}{p}}E_{\lambda}(V_{+})^{\frac{q}{p}}%
=m_{+}^{\frac{q}{p}}E_{\lambda}(u)^{\frac{p-q}{p}}, \quad \mbox{i.e.} \quad E_\lambda(u)^{1-\frac{q}{p}}\leq m.\]
Now, since $u$ solves $(P_{\lambda})$, we have $E_{\lambda}(u)=\int_{\Omega
}au^{q}$, so the latter inequality yields
\[
E_{\lambda}\left(  \left(  \int_{\Omega}au^{q}\right)  ^{-\frac{1}{q}%
}u\right)  =\frac{E_{\lambda}(u)}{\left(  \int_{\Omega}au^{q}\right)
^{\frac{p}{q}}}\leq m_{+},
\]
and by Proposition \ref{um}, we must have $\left(  \int_{\Omega}au^{q}\right)
^{-\frac{1}{q}}u\equiv V_{+}$. By Corollary \ref{un}, we deduce that $u$ and
$U_{+}$ are multiple of each other. Since both solve $(P_{\lambda})$ and
$\int_{\Omega}au^{q}>0$, Lemma \ref{la}(4) says that $u\equiv U_{+}$.
\end{proof}

In the next result we shall use the weak sub-supersolutions method. If
$\mathbf{B}u=u$ then we say that $0\leq u\in W^{1,p}\left(  \Omega\right)  $
is a \textit{supersolution }of $(P_{\lambda})$ whenever
\begin{equation}
\int_{\Omega}|\nabla u|^{p-2}\nabla u\nabla\phi\geq\int_{\Omega}\left(
\lambda u^{p-1}+au^{q-1}\right)  \phi\quad\mbox{for all } 0\leq\phi\in
W_{0}^{1,p}\left(  \Omega\right)  , \label{def}%
\end{equation}
and we say that $0\leq u\in W_{0}^{1,p}\left(  \Omega\right)  $ is a
\textit{subsolution} of $(P_{\lambda})$ whenever
\begin{equation}
\int_{\Omega}|\nabla u|^{p-2}\nabla u\nabla\phi\leq\int_{\Omega}\left(
\lambda u^{p-1}+au^{q-1}\right)  \quad\mbox{for all } 0\leq\phi\in W_{0}%
^{1,p}\left(  \Omega\right)  . \label{def2}%
\end{equation}
Note that we impose the subsolution to lie in $W_{0}^{1,p}\left(
\Omega\right)  $ since we are seeking for nonnegative solutions of $\left(
P_{\lambda}\right)  $. Similarly, if $\mathbf{B}u=\partial_{\nu}u$ then $0\leq
u\in W^{1,p}\left(  \Omega\right)  $ is a \textit{supersolution }of
$(P_{\lambda})$ whenever \eqref{def} holds for all $0\leq\phi\in
W^{1,p}\left(  \Omega\right)  $, and $0\leq u\in W^{1,p}\left(  \Omega\right)
$ is a \textit{subsolution }of $(P_{\lambda})$ if \eqref{def2} holds for all
$0\leq\phi\in W^{1,p}\left(  \Omega\right)  $.

For $f, g\in C(\overline{\Omega})$, we write $f<g$ if $f\leq g$ and
$f\not \equiv g$ in $\overline{\Omega}$.

\begin{proposition} 
\label{pin} Let $U_{+}(\lambda)$ be the unique nonnegative global minimizer of $I_\lambda$ for $\lambda<0$. Then:
\begin{enumerate}
	\item $U_{+}(\lambda)<U_{+}(\lambda^{\prime})$ if $\lambda<\lambda^{\prime}<0$.
	\item $U_+(\lambda) \to U_+(\lambda_0)$ in $X$ if $\lambda \to \lambda_0<0$. 
	\item $U_{+}(\lambda)\rightarrow0$ in $X$ as
	$\lambda\rightarrow-\infty$.
\end{enumerate}
\end{proposition}

\begin{proof}\strut
\begin{enumerate}
\item Let $\lambda<\lambda^{\prime}<0$. Then $U_{+}(\lambda)$ is a (strict) weak
subsolution of $(P_{\lambda^{\prime}})$. On the other hand, since
$\lambda^{\prime}<0$, any sufficiently large constant is a weak supersolution
larger than $U_{+}(\lambda)$. By the weak sub-supersolutions method, we find a
solution $u$ of $(P_{\lambda^{\prime}})$ with $u\geq U_{+}(\lambda)$. In
particular $u>0$ in $\Omega_{a}^{+}$ and, by Proposition \ref{us}, we find
that $u\equiv U_{+}(\lambda^{\prime})$, which yields the conclusion.
\item Let $\lambda_n \to \lambda_0<0$. Since  the sequence $U_n:=U_+(\lambda_n)$ stays
bounded in $C(\overline{\Omega})$, it is bounded in $X$, and up to a subsequence, we have $U_n \rightharpoonup U_0$ in $X$, for some $U_0\geq 0$. Taking $U_n-U_0$ as test function in $(P_{\lambda_n})$ we find that $U_n \to U_0$ in $X$. As $\{U_n\}$ is positive and bounded away from zero in $\Omega_a^+$, we have $U_0>0$ in $\Omega_a^+$. Finally, since $U_0$ solves $(P_{\lambda_0})$ and $\lambda_0<0$, by uniqueness we must have $U_0\equiv U_+(\lambda_0)$.
\item Arguing as in the previous item, we know that $U_{\lambda}=U_{+}(\lambda)$ stays
bounded in $C(\overline{\Omega})$, and thus in $X$, as $\lambda\rightarrow
-\infty$, so
\[
\int_{\Omega}U_{\lambda}^{p}=\frac{1}{\lambda}\left(  \int_{\Omega}|\nabla
U_{\lambda}|^{p}-\int_{\Omega}aU_{\lambda}^{q}\right)  \rightarrow0
\]
as $\lambda\rightarrow-\infty$. Hence $U_{\lambda}\rightarrow0$ in
$L^{p}(\Omega)$. Finally, since $\int_{\Omega}|\nabla U_{\lambda}|^{p}\leq
\int_{\Omega}aU_{\lambda}^{q}$ we have $U_{\lambda}\rightarrow0$ in $X$.
\end{enumerate}	
\end{proof}

\begin{remark}\label{rin}
In the Dirichlet case $(P_0)$ has a unique ground state solution, and this one is the only solution of $(P_0)$ positive in $\Omega_a^+$, cf. \cite[Theorem 1.1]{KRQUpp}. So the increasingness and continuity of $U_+(\lambda)$ hold in $(-\infty,0]$. The same conclusion applies to the Neumann problem if we assume $\int_\Omega a<0$.
\end{remark}

\begin{proposition}
\label{puu}\strut

\begin{enumerate}
\item Any nontrivial weak subsolution of $(P_{\lambda})$ belongs to
$\mathcal{A}_{+}$ for $\lambda\leq\lambda_{1}$.

\item If $\Omega_{a}^{+}$ is connected then $U_{+}$ is the unique nontrivial
solution of $(P_{\lambda})$ for $\lambda<0$. The same conclusion holds for
$\lambda=0$ if $\int_{\Omega}a<0$ when $X=W^{1,p}(\Omega)$.
\end{enumerate}
\end{proposition}

\begin{proof}\strut
	\begin{enumerate}
		\item We use the relation $0\leq E_{\lambda}%
		(u)\leq\int_{\Omega}au^{q}$, which holds for any weak subsolution with $\lambda
		\leq\lambda_{1}$. Moreover, $E_\lambda(u)=0$ only if $\lambda=\lambda_1$ and $u$ is an eigenfunction associated to $\lambda_1$, which is not possible since such eigenfunctions are positive in $\Omega$ and $a^- \not \equiv 0$.
		
	\item If $\Omega_{a}^{+}$ is connected and $u\not \equiv 0$
	solves $(P_{\lambda})$ with $\lambda\leq 0$ then, since $u\in\mathcal{A}_{+}$,
	we have $u\not \equiv 0$ in $\Omega_{a}^{+}$. By the strong maximum principle \cite{Va},
	we see that $u>0$ in $\Omega_{a}^{+}$. If $\lambda<0$ then Proposition \ref{us} yields the
	conclusion. For $\lambda=0$ we make use of \cite[Proposition 2.7]{KRQUpp}.
	\end{enumerate}
\end{proof}

\section{Asymptotics and positivity results}

Under the condition $\int_{\Omega}a\phi_{1}^{q} > 0$ we show that
$U_{+}(\lambda)$ blows up and we provide an asymptotic expression for it as
$\lambda\to\lambda_{1}^{-}$. A similar result holds for $U_{-}(\lambda)$. As a
consequence we can deduce their positivity for $\lambda$ close to $\lambda
_{1}$.

\begin{theorem}
\label{t3} \strut

\begin{enumerate}
\item If $\int_{\Omega}a\phi_{1}^{q}>0$ then $U_{+}(\lambda)\sim m_{+}
(\lambda)^{-\frac{1}{p-q}}\left(  \int_{\Omega}a\phi_{1}^{q}\right)
^{-\frac{1}{q}}\phi_{1}$ in $C^{1}(\overline{\Omega})$ as $\lambda
\rightarrow\lambda_{1}^{-}$, i.e.
\[
m_{+}(\lambda)^{\frac{1}{p-q}}U_{+}(\lambda)\rightarrow\left(  \int_{\Omega
}a\phi_{1}^{q}\right)  ^{-\frac{1}{q}}\phi_{1}\quad\mbox{in }C^{1}
(\overline{\Omega})\quad\mbox{as }\lambda\rightarrow\lambda_{1}^{-}.
\]
Moreover $m_{+}$ is continuous with respect to $\lambda$ and $m_{+}
(\lambda_{1} )=0$, so that $\min_{K} U_{+}(\lambda) \to\infty$ for any compact
$K \subset\Omega$. In particular, $U_{+}(\lambda) \in\mathcal{P}^{\circ}$ for
$\lambda$ close enough to (and smaller than) $\lambda_{1}$.

\item If $\int_{\Omega}a\phi_{1}^{q}<0$ then $U_{-}(\lambda)\sim
(-m_{-}(\lambda))^{-\frac{1}{p-q}}\left(  -\int_{\Omega}a\phi_{1}^{q}\right)
^{-\frac{1}{q}}\phi_{1}$ in $C^{1}(\overline{\Omega})$ as $\lambda
\rightarrow\lambda_{1}^{+}$, i.e.
\[
(-m_{-}(\lambda))^{\frac{1}{p-q}}U_{-}(\lambda)\rightarrow\left(
-\int_{\Omega}a\phi_{1}^{q}\right)  ^{-\frac{1}{q}}\phi_{1}\quad
\mbox{in }C^{1}(\overline{\Omega})\quad\mbox{as }\lambda\rightarrow\lambda
_{1}^{+},
\]
and similar statements as in (1) hold for $U_{-}$.
\end{enumerate}
\end{theorem}

Theorem \ref{t3} is a consequence of Corollary \ref{un} and the following result:

\begin{proposition}
\strut

\begin{enumerate}
\item If $\int_{\Omega}a\phi_{1}^{q}>0$ then $V_{+}(\lambda)\rightarrow\left(
\int_{\Omega}a\phi_{1}^{q}\right)  ^{-\frac{1}{q}}\phi_{1}$ in $C^{1}%
(\overline{\Omega})$ as $\lambda\rightarrow\lambda_{1}^{-}$.

\item If $\int_{\Omega}a\phi_{1}^{q}<0$ then $V_{-}(\lambda)\rightarrow\left(
-\int_{\Omega}a\phi_{1}^{q}\right)  ^{-\frac{1}{q}}\phi_{1}$ in $C^{1}%
(\overline{\Omega})$ as $\lambda\rightarrow\lambda_{1}^{+}$.
\end{enumerate}
\end{proposition}

\begin{proof}
We shall prove only (1), since the proof of (2) is similar. Recall that
$\int_{\Omega}a\phi_{1}^{q}>0$ implies that $\lambda^{\ast}=\lambda_{1}$. Note
also that $\lambda\mapsto m_{+}(\lambda)$ is continuous (since it is a concave
map) and $m_{+}(\lambda_{1})=0$. Let $\lambda_{n}\rightarrow\lambda_{1}^{-}$
and $v_{n}:=V_{+}(\lambda_{n})$. Then $E_{\lambda_{n}}(v_{n})=m_{+}%
(\lambda_{n})\rightarrow0$. By Lemma \ref{la}(2) we know that $\{v_{n}\}$ is
bounded in $X$, and we may assume that $v_{n}\rightharpoonup v_{0}$ in $X$. We
find that $\int_{\Omega}av_{0}^{q}=1$ and
\[
E_{\lambda_{1}}(v_{0})\leq\lim E_{\lambda_{n}}(v_{n})=0\leq E_{\lambda_{1}%
}(v_{0}),
\]
i.e. $E_{\lambda_{n}}(v_{n})\rightarrow E_{\lambda_{1}}(v_{0})=0$. It follows
that $v_{n}\rightarrow v_{0}$ in $X$ and $v_{0}=c\phi_{1}$ for some $c>0$.
From $\int_{\Omega}av_{0}^{q}=1$ we infer that $c=\left(  \int_{\Omega}%
a\phi_{1}^{q}\right)  ^{-\frac{1}{q}}$. Finally, by standard regularity (e.g.
\cite[Lemma 2.3]{AR}), we obtain the convergence in $C^{1}(\overline{\Omega}%
)$.
\end{proof}

\begin{remark}
\label{ras} \strut

\begin{enumerate}
\item A `bifurcation from infinity' result as Theorem \ref{t3} has been
established in \cite{alama,B} for $p=2$. Moreover, in \cite{alama} the author
also deduces that the bifurcating solutions are positive. However, these works
do not provide the asymptotics of the bifurcating solutions. Note also that
such result does not hold in the indefinite superlinear case ($p=2$), under
some further restrictions on $a$ and $q$, in view of the \textit{a priori}
bounds established in \cite{ALG,BCN} for positive solutions of $(P_{\lambda})$.

\item Note that for the Neumann problem we have $\phi_{1}>0$ on $\overline
{\Omega}$, so $\displaystyle \min_{\overline{\Omega}}U_{\pm}(\lambda)
\to\infty$ as $\lambda\to 0^{\mp}$ if $\int_{\Omega}a \gtrless0$.
\end{enumerate}
\end{remark}

We prove now that nonnegative minimizers associated to $m_{\pm}$ are positive
when $q$ is close to $p$. To be more precise, they lie in $\mathcal{P}^{\circ
}$.

First we need the following result:

\begin{proposition}
\label{le} $\lambda^{*}(p)= \displaystyle \lim_{q \to p^{-}} \lambda^{*}(q)$
\end{proposition}

\begin{proof}
Let $q_{n}\rightarrow p^{-}$ and $\lambda^{\ast}(q_{n})=\int_{\Omega}|\nabla
u_{n}|^{p}$ with $\int_{\Omega}|u_{n}|^{p}=1$ and $\int_{\Omega}%
a|u_{n}|^{q_{n}}\geq0$. Then $\{u_{n}\}$ is bounded in $X$, so we can assume
that $u_{n}\rightharpoonup u_{0}$ in $X$, for some $u_{0}$. It follows that
$\int_{\Omega}|u_{0}|^{p}=1$ and $\int_{\Omega}au_{0}^{p}=\lim\int_{\Omega
}au_{n}^{q_{n}}\geq0$, so that
\[
\lambda^{\ast}(p)\leq\int_{\Omega}|\nabla u_{0}|^{p}\leq\liminf\int_{\Omega
}|\nabla u_{n}|^{p}=\liminf\lambda^{\ast}(q_{n}).
\]
Let now $u$ be such that $\lambda^{\ast}(p)=\int_{\Omega}|\nabla u|^{p}$,
$\int_{\Omega}|u|^{p}=1$ and $\int_{\Omega}a|u|^{p}\geq0$. We choose some
$B\subset\Omega_{a}^{+}$ and $\psi\in C_{0}^{1}(B)$ such that $\psi>0$ in $B$.
Then $\int_{\Omega}a|u+t\psi|^{p}>0$ for any $t>0$, so that $\lim\int_{\Omega
}a|u+t\psi|^{q_{n}}>0$ for any $t>0$. It follows that $\limsup\lambda^{\ast
}(q_{n})\leq\frac{\int_{\Omega}|\nabla(u+t\psi)|^{p}}{\int_{\Omega}%
|(u+t\psi)|^{p}}$ for any $t>0$. Letting $t\rightarrow0$, we obtain
$\limsup\lambda^{\ast}(q_{n})\leq\lambda^{\ast}(p)$, which concludes the
proof.
\end{proof}

Since $q$ is not fixed anymore, we shall write $m_{\pm}=m_{\pm}(q)$.

\begin{proposition}
\label{cs} For any $\lambda<\lambda^{\ast}(p)$ there exists $q_{0}%
=q_{0}(\lambda,a)\in\lbrack1,p)$ such that any nonnegative minimizer
associated to $m_{+}(q)$ belongs to $\mathcal{P}^{\circ}$ for $q\in(q_{0},p)$.
The same conclusion holds for nonnegative minimizers associated to $m_{-}(q)$
if $\int_{\Omega}a\phi_{1}^{p}<0$ and $\lambda_{1}<\lambda<\lambda^{\ast}(p)$.
Moreover, one can choose $q_{0}(\lambda,a)$ as a decreasing function of
$\lambda$ for $\lambda<0$.


\end{proposition}

\begin{proof}
Let $\lambda<\lambda^{\ast}(p)$. By Proposition \ref{le} we know that
$\lambda<\lambda^{\ast}(q)$ for $q$ close enough to $p$. Assume by
contradiction that there exists a sequence $q_{n}\rightarrow p^{-}$ and a
sequence $\{v_{n}\}\subset\mathcal{S}_{+}$ with $0\leq v_{n}\not \in
\mathcal{P}^{\circ}$ and $E_{\lambda}(v_{n})=m_{n}:=m_{+}(q_{n})$. Let us fix
$\phi\in X$ such that $\int_{\Omega}a\phi^{p}=1$. Then $\int_{\Omega}%
a\phi^{q_{n}} \to1$, and it follows that
\[
0<m_{n}\leq E_{\lambda}\left(  \left(  \int_{\Omega}a\phi^{q_{n}}\right)
^{-\frac{1}{q_{n}}}\phi\right)  \rightarrow E_{\lambda}(\phi),
\]
i.e. $\{m_{n}\}$ is bounded.
By Lemma \ref{la}(2) we know that $\left\{  v_{n}\right\}  $ is bounded in
$X$, and we can assume that $v_{n}\rightharpoonup v_{0}$ in $X$, and
$\int_{\Omega}av_{0}^{p}=\lim\int_{\Omega}av_{n}^{q_{n}}=1$, so that
$v_{0}\not \equiv 0$. Note also that $v_{n}$ solves
\[
-\Delta_{p}v_{n}=\lambda v_{n}^{p-1}+m_{n}av_{n}^{q_{n}-1},\quad v_{n}\in X.
\]
Taking $v_{n}-v_{0}$ as test function in this equation, we see that
\[
\int_{\Omega}|\nabla v_{n}|^{p-2}\nabla v_{n}\nabla(v_{n}-v_{0})\rightarrow0.
\]
By the $(S_{+})$ property of the $p$-Laplacian, we deduce that $v_{n}%
\rightarrow v_{0}$ in $X$. Since $\{m_{n}\}$ is bounded, we can assume that
$m_{n}\rightarrow m_{0}$ for some $m_{0}\geq0$. Taking the limit in the above
equation, we see that $v_{0}\geq0$ solves
\[
-\Delta_{p}v_{0}=\lambda v_{0}^{p-1}+m_{0}av_{0}^{p-1},\quad v_{0}\in X.
\]
Since the strong maximum principle applies to this equation, we infer that
$v_{0}\in\mathcal{P}^{\circ}$. Finally, by elliptic regularity, we find that
$v_{n}\rightarrow v_{0}$ in $C^{1}(\overline{\Omega})$ and consequently
$v_{n}\in\mathcal{P}^{\circ}$ for $n$ large enough, which yields a
contradiction. Therefore, setting
\[
q_{0}(\lambda):=\inf\{q_{0}\in
(1,p):\mbox{nonnegative minimizers for $m_+(q)$ lie in }\mathcal{P}^{\circ
}\ \forall q\in(q_{0},p)\},
\]
we have proved that $1\leq q_{0}(\lambda)<p$. Thanks to Corollary \ref{un},
for $\lambda<0$ we have
\[
q_{0}(\lambda)=\inf\{q_{0}\in(1,p):U_{+}(\lambda,q)\in\mathcal{P}^{\circ
}\ \forall q\in(q_{0},p)\}.
\]
Moreover, if $\lambda<\lambda^{\prime}<0$ then, by Proposition \ref{pin},
$U_{+}(\lambda,q)\leq U_{+}(\lambda^{\prime},q)$, so that $U_{+}%
(\lambda^{\prime},q)\in\mathcal{P}^{\circ}$ for $q_{0}(\lambda)<q<p$. It
follows that $q_{0}(\lambda^{\prime})\leq q_{0}(\lambda)$.

Assume now that $\int_{\Omega}a\phi_{1}^{q}<0$, and $\{v_{n}\}\subset
\mathcal{S}_{-}$, with $0 \leq v_{n}\not \in \mathcal{P}^{\circ}$ and
$E_{\lambda}(v_{n})=m_{n}:=m_{-}(\lambda,q_{n})$ for every $n$. Since
$\lambda_{1}<\lambda<\lambda^{\ast}(p)$ we have $m_{n}<0$ for every $n$. From
Lemma \ref{la}(2) we deduce that $\{v_{n}\}$ is bounded in $X$. Arguing as in
the first part of the proof we reach a contradiction. 
\end{proof}

Let us prove now that $U_{+}$ becomes positive as $a^{-}$ approaches zero. We
fix $\lambda<\lambda_{1}$, $q$ and $a$. Set, for $\delta>0$, $a_{\delta
}:=a^{+}-\delta a^{-}$, and
\[
M_{\delta}:=\inf\left\{  I_{\delta}(u):u\in X,\int_{\Omega}a_{\delta}%
|u|^{q}>0\right\}  ,
\]
where
\[
I_{\delta}(u):=\frac{1}{p}E_{\lambda}(u)-\frac{1}{q}\int_{\Omega}a_{\delta
}(x)|u|^{q}%
\]
for $u\in X$.

\begin{proposition}
\label{pa} Let $\lambda<\lambda_{1}$ and $a\in C(\overline{\Omega})$. There
exists $\delta_{0}>0$ such that any nonnegative minimizer associated to
$M_{\delta}$ belongs to $\mathcal{P}^{\circ}$ for $0<\delta<\delta_{0}$.
\end{proposition}

\begin{proof}
For $\delta>0$ small enough we have $\int_{\Omega}a_{\delta}\phi_{1}^{q}\geq
0$, and thus $\lambda^{\ast}(a_{\delta})=\lambda_{1}$. It follows that
$M_{\delta}$ is achieved by $U_{+}(\delta)$. Assume by contradiction that
$\delta_{n}\rightarrow0$ and $u_{n}:=U_{+}(\delta_{n})\not \in \mathcal{P}%
^{\circ}$. By Lemma \ref{la}(3) we have
\[
C_{\lambda}\Vert u_{n}\Vert^{p}\leq E_{\lambda}(u_{n})\leq\int_{\Omega}%
a^{+}u_{n}^{q}\leq C\Vert u_{n}\Vert^{q}%
\]
for some $C,C_{\lambda}>0$ and every $n$, so that $\{u_{n}\}$ is bounded in
$X$. We can assume that $u_{n}\rightharpoonup u_{0}$ in $X$. Taking
$u_{n}-u_{0}$ as test function in the equation solved by $u_{n}$, we obtain
that
\[
\int_{\Omega}|\nabla u_{n}|^{p-2}\nabla u_{n}\nabla(u_{n}-u_{0})\rightarrow0.
\]
By the $(S_{+})$ property of the $p$-Laplacian, we infer that $u_{n}%
\rightarrow u_{0}$ in $X$, and by elliptic regularity, we have $u_{n}%
\rightarrow u_{0}$ in $C^{1}(\overline{\Omega})$. Since $a_{n}\rightarrow
a^{+}$ in $C(\overline{\Omega})$, we see that $u_{0}$ solves
\[
-\Delta_{p}u_{0}=\lambda u_{0}^{p-1}+a^{+}(x)u_{0}^{q-1},\quad u_{0}%
\geq0,\quad u_{0}\in X.
\]
By the strong maximum principle, we have either $u_{0}\equiv0$ or $u_{0}%
\in\mathcal{P}^{\circ}$. Let us show that $u_{0}\not \equiv 0$, in which case
we obtain a contradiction with $u_{n}\not \in \mathcal{P}^{\circ}$. Recall
that $u_{n}$ is the unique nonnegative global minimizer of $I_{\delta_{n}}$.
Let $U$ be the unique positive global minimizer (cf. \cite{DS}) of
\[
I_{0}(u)=\frac{1}{p}E_{\lambda}(u)-\frac{1}{q}\int_{\Omega}a^{+}%
(x)|u|^{q},\quad u\in X.
\]
Then
\[
I_{0}(u_{0})=\lim I_{\delta_{n}}(u_{n})\leq\lim I_{\delta_{n}}(U)=I_{0}(U)<0,
\]
which shows that $u_{0}\equiv U\not \equiv 0$. The proof is complete. 
\end{proof}

\begin{remark}
It is clear that the previous result does not apply to $U_{-}$ since this
solution exists only if $\int_{\Omega}a\phi_{1}^{q}<0$, which prevents $a^{-}$
to be arbitrarily small. Similarly, it does not apply to $U_{+}$ if
$\lambda>\lambda_{1}$, since in this case, for $\delta>0$ small, we have
$\lambda_{1}=\lambda^{*}$ and consequently $M_{\delta}$ is not achieved (see
Remark \ref{nem}).
\end{remark}

\section{Dead core solutions}

We prove now Theorem \ref{dc}, which is a direct consequence of Proposition
\ref{213} below. Let us stress the dependence on $a$ and write $\lambda^{\ast
}(a)$, $\mathcal{A}_{+}(a)$, $U_{+}(a)$, and $\left(  P_{\lambda,a}\right)  $
instead of $\lambda^{\ast}$, $\mathcal{A}_{+}$, $U_{+}$, and $\left(
P_{\lambda}\right)  $ respectively ($q$ is fixed).

\begin{lemma}
\label{bound}Let $a\in C(\overline{\Omega})$ and $\lambda<\lambda^{\ast}(a)$.
There exists a constant $K=K(\lambda,a)>0$ such that any weak subsolution $u$
of $\left(  P_{\lambda,a}\right)  $ lying in $\mathcal{A}_{+}(a)$ satisfies
$\Vert u\Vert_{\infty}\leq K$.
\end{lemma}

\begin{proof}
If $u \in
\mathcal{A}_{+}(a)$ is a subsolution of $\left(  P_{\lambda,a}\right)  $ then by Lemma \ref{la} (3), we have
\[
C_{\lambda}(a)\Vert u\Vert^{p}\leq E_{\lambda}(u)\leq\int_{\Omega}au^{q}\leq
C(a)\Vert u\Vert^{q},
\]
for some $C_{\lambda}(a),C(a)>0$. Thus $\Vert u\Vert\leq\left(  C(a)C_{\lambda
}(a)^{-1}\right)  ^{\frac{1}{p-q}}$, and by a bootstrap argument we get the
conclusion.
\end{proof}

Let $\lambda_{\ast}=\lambda_{\ast}\left(  a\right)  :=\inf\left\{
\int_{\Omega}|\nabla u|^{p}:u\in X,\int_{\Omega}|u|^{p}=1,a^{-}u\equiv
0\right\}  $. Note that $\lambda_{1}\leq\lambda^{\ast}(a)\leq\lambda_{\ast
}(a)$. Recall that $\Omega_{a}^{-}:=\{x\in\Omega:a<0\}$. The next result is
based on the comparison of a given solution with a local barrier function
vanishing in a prescribed set.

\begin{proposition}
\label{213}Assume that $a\in C(\overline{\Omega})$ changes sign,
$\Omega^{\prime}\Subset\Omega_{a}^{-}$, and $a_{n}:=a^{+}-na^{-}$. There
exists $n_{0}=n_{0}\left(  a,p,q,N,\Omega^{\prime},\lambda\right)  >0$ such
that any solution of $\left(  P_{\lambda,a_{n}}\right)  $ lying in
$\mathcal{A}_{+}\left(  a_{n}\right)  $ vanishes in $\Omega^{\prime}$ for all
$n\geq n_{0}$, provided one of the following conditions hold:

\begin{enumerate}
\item $\lambda\leq0$ (in this case $n_{0}$ does not depend on $\lambda$).

\item $p=2$ and $\lambda<\lambda_{\ast}\left(  a\right)  $.
\end{enumerate}
\end{proposition}

\textit{Proof}. The proof is inspired by the computations in \cite[Theorem
3.2]{ans}.

\begin{enumerate}
\item Let $B=B_{R}(x_{0})$ be a ball with $B\Subset\Omega_{a}^{-}$. Let $n>0$
be large so that $\int_{\Omega}a_{n}\phi_{1}^{q}<0$ (and so, $\lambda^{\ast
}\left(  a_{n}\right)  >0$), and let $u_{n}\in\mathcal{A}_{+}\left(
a_{n}\right)  $ be a solution of $\left(  P_{\lambda,a_{n}}\right)  $. Then
$\mathcal{L}u_{n}:=-\Delta_{p}u_{n}-\lambda u_{n}^{p-1}=-na^{-}u_{n}^{q-1}$ in
$B$. For $k>0$ and $\beta>1$ to be chosen later, let us define
\[
w(x):=k\left(  \left\vert x-x_{0}\right\vert ^{2}-t^{2}\right)  ^{\beta}%
\quad\text{where}\quad0\leq t<R,\quad t\leq\left\vert x-x_{0}\right\vert \leq
R.
\]
In order to avoid overloading the notation, we set $r:=\left\vert
x-x_{0}\right\vert $ and write $w=w\left(  r\right)  =k\left(  r^{2}%
-t^{2}\right)  ^{\beta}$. We have that
\[
\left\vert \nabla w\right\vert ^{p-2}\nabla w=\left(  2k\beta\right)
^{p-1}\left(  r^{2}-t^{2}\right)  ^{\left(  \beta-1\right)  \left(
p-1\right)  }r^{p-2}\left(  x-x_{0}\right)
\]
and so, after some computations we find that
\begin{align*}
\operatorname{div}\left(  \left\vert \nabla w\right\vert ^{p-2}\nabla
w\right)   &  =\left(  2k\beta\right)  ^{p-1}\left(  r^{2}-t^{2}\right)
^{\beta\left(  p-1\right)  -p}r^{p-2}[2\left(  \beta-1\right)  \left(
p-1\right)  r^{2}\\
&  +\left(  p-2+N\right)  \left(  r^{2}-t^{2}\right)  ].
\end{align*}
Thus, since $\lambda\leq0$, for $t\leq r\leq R$ we get%
\begin{align}
\mathcal{L}w  &  \geq-\Delta_{p}w\label{lapa}\\
&  \geq-\left(  2k\beta\right)  ^{p-1}(r^{2}-t^{2})^{\beta\left(  p-1\right)
-p}R^{p-2}[2\left(  \beta-1\right)  \left(  p-1\right)  R^{2}\nonumber\\
&  +\left(  p-2+N\right)  (R^{2}-t^{2})].\nonumber
\end{align}
Let us now choose%
\begin{align}
\beta &  :=\frac{p}{p-q},\quad0<\underline{a}:=\min_{\overline{B}}%
a^{-},\mbox{ and }\label{lapa2}\\
k  &  :=\left(  \frac{n\underline{a}}{\left(  2\beta\right)  ^{p-1}%
R^{p-2}\left(  2\left(  \beta-1\right)  \left(  p-1\right)  R^{2}+\left(
p-2+N\right)  \left(  R^{2}-t^{2}\right)  \right)  }\right)  ^{\frac{1}{p-q}%
}.\nonumber
\end{align}
We observe that $\beta\left(  p-1\right)  -p=\beta\left(  q-1\right)  $.
Taking into account this, (\ref{lapa}) and (\ref{lapa2}), one can check that
$\mathcal{L}w\geq-na^{-}w^{q-1}$ in $B\diagdown B_{t}(x_{0})$. We now extend
$w$ by zero to the whole $\overline{B}$. Then $w\in C^{1}(\overline{B})$,
$w^{\prime}\left(  t\right)  =0$ and hence, in the weak sense, $\mathcal{L}%
w\geq-na^{-}w^{q-1}$ in $B$. Also, by (\ref{lapa2}) and Lemma \ref{bound}
(applied with $a_{n}$ and $0$ in place of $a$ and $\lambda$, respectively),
enlarging $n$ if necessary, we may assume that $w>u_{n}$ on $\partial B$. We
claim that, for such $n$, $w\geq u_{n}$ in $B$. Indeed, if this does not
happen then $\mathcal{O}:=\left\{  x\in B:w\left(  x\right)  <u_{n}\left(
x\right)  \right\}  \not =\emptyset$. Since $w>u_{n}$ on $\partial B$ it holds
that $w=u_{n}$ on $\partial\mathcal{O}$. Also, $\mathcal{L}w\geq
\mathcal{L}u_{n}$\ in $\mathcal{O}$. Now, since $\lambda\leq0$, the weak
comparison principle (see e.g. \cite[Theorem 1.2(a)]{Dam}) applies to
$\mathcal{L}$ in $\mathcal{O}$ and says that $w\geq u_{n}$\ in $\mathcal{O}$,
a contradiction. Hence, $w\geq u_{n}$ in $B$ as claimed and, in particular,
$u_{n}\equiv0$ in $B_{t}(x_{0})$.\newline Finally, let $\Omega^{\prime}%
\Subset\Omega_{a}^{-}$ and set $\delta:=\,$dist$\,(\overline{\Omega^{\prime}%
},\partial\Omega_{a}^{-})>0$. Take a finite covering of $\Omega^{\prime}$ of
open balls $B_{x_{i}}\left(  t\right)  $, where $t:=\delta/2$. Now set
$R:=t+\varepsilon$, with $0<\varepsilon<\delta/2$. Then $\Omega^{\prime
}\subset\cup B_{x_{i}}\left(  t\right)  $ and $\cup B_{x_{i}}\left(  R\right)
\Subset\Omega_{a}^{-}$, and hence item (1) follows from the above part of the proof.

\item We proceed similarly. We may assume that $0<\lambda<\lambda_{\ast
}\left(  a\right)  $. We first observe that $\lambda^{\ast}(a_{n}%
)\rightarrow\lambda_{\ast}\left(  a\right)  $ when $n\rightarrow\infty$.
Indeed, first note that $\lambda_{1}\leq\lambda^{\ast}(a_{n})\leq\lambda
_{\ast}\left(  a\right)  $ for every $n$. Let $z_{n}\geq0$ achieve
$\lambda^{\ast}(a_{n})$, i.e.
\[
\int_{\Omega}|\nabla z_{n}|^{2}=\lambda^{\ast}(a_{n}),\quad\int_{\Omega}%
z_{n}^{2}=1,\quad\mbox{and}\quad\int_{\Omega}a_{n}z_{n}^{q}\geq0.
\]
Then $\{z_{n}\}$ is bounded in $X$, so we may assume that $z_{n}%
\rightharpoonup z_{0}$ in $X$, for some $z_{0}$. From $\int_{\Omega}a_{n}%
z_{n}^{q}\geq0$ we deduce that
\[
\int_{\Omega}a^{-}z_{n}^{q}\leq\frac{1}{n}\int_{\Omega}a^{+}z_{n}^{q}\leq
\frac{C}{n},
\]
so that $\int_{\Omega}a^{-}z_{0}^{q}=0$. Hence $a^{-}z_{0}\equiv0$, which
yields
\[
\lambda_{\ast}\left(  a\right)  \leq\int_{\Omega}|\nabla z_{0}|^{2}\leq
\lim\lambda^{\ast}(a_{n})\leq\lambda_{\ast}\left(  a\right)  ,
\]
i.e. $\lambda_{\ast}\left(  a\right)  =\lim\lambda^{\ast}(a_{n})$. Let
$u_{n}\in\mathcal{A}_{+}\left(  a_{n}\right)  $ be a solution of $\left(
P_{\lambda,a_{n}}\right)  $. Then $\mathcal{L}u_{n}:=-\Delta u_{n}-\lambda
u_{n}=-na^{-}u_{n}^{q-1}$ in $\Omega_{a}^{-}$. By the above paragraph, taking
$n$ large enough we may assume that $\lambda<\lambda^{\ast}\left(
a_{n}\right)  $. Now, let $B=B_{R}(x_{0})\Subset\Omega_{a}^{-}$ be a ball with
$R>0$ small, so that $\lambda<\lambda_{1}\left(  B\right)  $. Define $w\in
C^{1}(\overline{B})$ as in item (1), with $\beta$ and $\underline{a}$ as in
(\ref{lapa2}), and%
\begin{equation}
k:=\left(  \frac{n\underline{a}}{2\beta\left(  \left(  \beta-1\right)
2R^{2}+N\left(  R^{2}-t^{2}\right)  \right)  +\lambda\left(  R^{2}%
-t^{2}\right)  ^{2}}\right)  ^{\frac{1}{2-q}}. \label{k2}%
\end{equation}
After some computations we obtain, for $t\leq r\leq R$,%
\begin{align}
\mathcal{L}w  &  =-k\left(  r^{2}-t^{2}\right)  ^{\beta-2}\left[
2\beta\left(  \left(  \beta-1\right)  2r^{2}+N\left(  r^{2}-t^{2}\right)
\right)  +\lambda\left(  r^{2}-t^{2}\right)  ^{2}\right] \label{l2}\\
&  \geq-k\left(  r^{2}-t^{2}\right)  ^{\beta-2}\left[  2\beta\left(  \left(
\beta-1\right)  2R^{2}+N\left(  R^{2}-t^{2}\right)  \right)  +\lambda\left(
R^{2}-t^{2}\right)  ^{2}\right]  .\nonumber
\end{align}
Hence, since $\beta\left(  q-1\right)  =\beta-2$, taking into account
(\ref{k2}) and (\ref{l2}) we infer that $\mathcal{L}w\geq-na^{-}wu^{q-1}$ in
weak sense in $B$. Now, recalling that $\lambda<\min\left(  \lambda^{\ast
}\left(  a_{n}\right)  ,\lambda_{1}\left(  B\right)  \right)  $ and
$\lambda_{1}\left(  B\right)  <\lambda_{1}\left(  \mathcal{O}\right)  $ we may
use Lemma \ref{bound} and apply the weak maximum principle to $\mathcal{L}$ in
$\mathcal{O}$. So, as in (1) we deduce that, enlarging $n$ if necessary,
$w\geq u_{n}$ in $B$ and therefore $u_{n}\equiv0$ in $B_{t}(x_{0})$. The rest
of the proof follows as in (1).\qed

\end{enumerate}


\begin{remark}
One can easily see from the proof above that Theorem \ref{dc} and Proposition
\ref{213} are in fact true for nonnegative weak subsolutions of $\left(
P_{\lambda,a_{n}}\right)  $.
\end{remark}

\end{document}